\documentclass[11pt]{article}
\usepackage{here}
\usepackage{amsmath,amssymb,amsfonts,amsthm}
\usepackage{cleveref}
\usepackage{pdflscape}
\usepackage{natbib}
\usepackage{enumerate}

\usepackage{tikz}
\usetikzlibrary {shapes}
\usetikzlibrary {arrows}
\usetikzlibrary {positioning}
\usetikzlibrary {calc}
\usetikzlibrary{fit}					
\newcommand{\nodeset}{\mathcal{N}}

\newcommand{\iso}{\cong}
\newcommand{\inj}{\mathrm{inj}}

\newcommand{\aut}{\mathrm{aut}}
\newcommand{\sub}{\mathrm{sub}}

\newcommand{\E}{\mathrm{E}}

\newcommand{\pa}{\mathrm{pa}}

\newcommand{\cip}{\mbox{\,$\perp\!\!\!\perp$\,}}

\newcommand{\cd}{\,|\,}

\newcommand{\skel}{\mathrm{sk}}

\newcommand{\ci}{\mbox{\protect $\: \perp \hspace{-2.3ex}
\perp$ }}

\newcommand{\n}[0]{\hspace*{.35em}}

\newcommand{\nn}[0]{\hspace*{.7em}}

\newcommand{\node}{\mbox {\LARGE
{$\mbox{$\circ$}$}}}

\newtheorem{prop}{Proposition}
\newtheorem{coro}{Corollary}

\newtheorem{lemma}{Lemma}

\newtheorem{expl}{Example}
\newcommand{\halm}{\hspace*{\fill} $\Box$\par}
\newenvironment{example}{\begin{expl}\rm}{\halm\end{expl}}

\title{Random Networks, Graphical Models,\\ and Exchangeability}
\author{Steffen Lauritzen\thanks{Steffen Lauritzen, Department of Mathematical Sciences, University of Copenhagen, Universitetsparken 5, 2100, Copenhagen, Denmark. email: \texttt{lauritzen@math.ku.dk}}\\University of Copenhagen \and 
Alessandro Rinaldo\\
Carnegie Mellon University\and
{Kayvan Sadeghi}\\
University of Cambridge}

\begin{document}
\maketitle
\begin{abstract} 
We study conditional independence relationships for random
    networks and their interplay with exchangeability.
    We show that, for finitely exchangeable network models, the empirical subgraph densities are
    maximum likelihood estimates of their theoretical counterparts.
    We then characterize all possible Markov structures for finitely exchangeable random
    graphs, thereby identifying a new class of Markov network models
    corresponding to bidirected Kneser graphs. In particular, we demonstrate
    that the fundamental
    property of  dissociatedness corresponds to a Markov property for
    exchangeable networks described by
    bidirected line graphs.
    Finally we study those exchangeable models that are also
    summarized in the sense that the probability of a network only depends on
    the degree distribution, and identify a class of models that is dual to the
    Markov graphs of \cite{fra86}. Particular emphasis is placed on studying consistency properties of network models under the process of forming subnetworks and  we show that the only consistent systems of Markov properties correspond to the empty graph, the bidirected line graph of the complete graph, and the complete graph.\\

\noindent \textbf{Keywords:} bidirected Markov property, de Finetti's theorem, exchangeable arrays, exponential random graph model, graph limit, graphon, Kneser graph, marginal beta model, M{\"o}bius parametrization, Petersen graph.
\end{abstract}

\section{Introduction}
Over the last decades, the collection and rapid diffusion of network data
originating from a wide spectrum
 of scientific areas have created the need for new statistical theories and
methodologies for modeling and analyzing large random graphs.
There now exists a very  large and rapidly growing body of literature on network
analysis: see, for example, \cite{Kolaczyk:2009}, \cite{Newman:2010},
and references therein.
Since the seminal contributions of \cite{fra86} and \cite{hol81},
the field of statistical network modeling has
advanced significantly and researchers have now gained a  much broader understanding of the
properties of network data and the many open challenges associated to
network modeling.

A common feature shared by many network models is that
of invariance to
the relabeling of the network units, or (finite) exchangeability, whereby  isomorphic graphs
have the same probabilities, and are therefore regarded as statistically equivalent.
Exchangeability is a basic form of probabilistic invariance, but also
 a natural and convenient simplifying assumption to impose when formalizing statistical models for
random graphs. Examples of popular network models which rely on exchangeability
include
many exponential random graph models, the stochastic block model, graphon-based models, latent space
models, to name a few.

In the
probability literature, exchangeability of infinite binary arrays --- which
encompasses
networks  --- is very well understood: see for example \cite{aldous:81,aldous:85,lauritzen:03,lauritzen:08}; and
\cite{kallenberg:05}. Of particular
relevance here is the work of \cite{diaconis:janson:08} \citep[but see
also][]{orbantz:roy:15}, which details the connections between
exchangeability of random graphs and the notion of graph limits developed in
\cite{lov06}.
While the existing results provide a canonical, \emph{analytic} representation of
infinite exchangeable networks  (see later in Section~\ref{sec:prelim}), they
appear to be too implicit for statistical modeling purposes. Indeed,
the properties implied by such representation
do not directly yield simple parametric statistical models for
finite networks.
Furthermore, while an exchangeable probability distribution for an infinite
networks by construction induce consistent families of exchangeable
distributions for all
its finite subnetworks, there is no guarantee that an exchangeable distribution
for networks of a given size may be extendable to a probability distribution for larger
(possibly infinite)
networks, or that the induced distributions share
the similar properties.  Thus, {\it a priori\/} any statistical model for finite exchangeable
networks that does not conform to the analytic representation of infinite
exchangeable networks need not be compatible in any 
meaningful  way with
models of the same type over networks of different sizes. These issues are
of great significance, as they render very difficult to determine the extent to which statistical inference
based on a subnetwork applies to a larger network, as argued in
\cite{CD:15}, and \cite{shalizi:rinaldo:13}.

The main goal of our article is to reveal and point out similarities and connections between seemingly unrelated concepts from different parts of the literature --- in particular graphical models, network models, and notions of exchangeability --- and our analysis is attempting to enhance understanding of the models and their properties.

We rely on the theory of exponential families and, especially, of
graphical models \citep{lau96} in order to categorize Markov properties
 implied by the assumption of  exchangeability, and the
 statistical models that can be derived from it. The idea of modeling networks
 as Markov random fields on binary variables is not new: it was originally suggested by
 \cite{fra86}, who argue that a particular class of exponential random graph models (to be
 discussed  below in Section \ref{sec:indep.structures})
 is able to capture what the authors, somewhat arbitrarily, posit as a natural form of dependence structure for
 networks, as well as the symmetries implied by
 exchangeability.  In contrast,  our analysis of the Markov
 properties of exchangeable network models is both principled, as it only relies
 on the assumption of exchangeability, and exhaustive, as it produces a complete
 list of all Markov properties expressed by exchangeable network models.
 We make the following specific contributions:
 \begin{itemize}
     \item We describe an exponential family representation for exchangeable
     	 network
     	 models where the sufficient statistics are given by all subgraph
     	 counts. The subgraph frequencies are maximum likelihood estimators
	 of the corresponding mean value parameters, which are
	 the marginal
	 probabilities of all subgraphs.
     \item  We study the Markov properties of finitely exchangeable network
     	 models, and show that such models are compatible only with four types
     	 of non-trivial conditional independence structures.

     \item We demonstrate that statistical models for finite networks induced by
     	 extremal exchangeable distributions on an infinite network
     	 are Markov with respect to a distinguished bidirected graphical model for marginal
     	 independence. 

     \item Finally, we propose a novel class of models for exchangeable networks
	 that are also {\it summarized,} in the sense that their properties  only depend on the
     	 degree distribution of the observed network.
 \end{itemize}
We should emphasize that the scope of this paper has been  strictly limited to the study of full exchangeablity and Markov structures in the sense used for standard graphical models and, as we show, the assumption of full exchangeabiliy, extendibility, and Markov properties in this sense is very restrictive. 

There are important types of Markov properties for random networks that we have not considered in this paper, notably that of \emph{partial conditional independence} \citep{snijders:06}, also used in \cite{hunter:etal:08}.  This concept allows the conditional independence properties of specific ties in a network to depend on the configuration of the remaining network. This type of Markov property is common in the analysis of spatial point processes, where it is referred to as \emph{data dependent Markov neighbourhoods} \citep{baddeley:moller:89}; indeed, a random network can be considered as a point process on the space of possible ties.  Also in the literature on Bayesian networks and graphical models, an analogous Markov theory has evolved and here the term \emph{context-specific independence} is often used, see e.g.\ \cite{Boutilier:etal:96} or \cite{nyman:etal:14} for a recent discussion.

Similarly, it is of interest but outside the scope of the present paper to discuss other variants of exchangeability, where we in particular mention \emph{partial exchangeability}  or \emph{block exchangeability} as used in \cite{schweinberger:handcock:15} or models derived from exchangeable measures, as in \cite{caron:fox:17}.

The structure of our paper is as follows. In Section~\ref{sec:prelim} we develop the necessary terminology and detail the concepts of exchangeable arrays, random networks, bidirected Markov properties, and important parametrizations of the models. In Section~\ref{sec:5} we provide parametrizations for exchangeable networks and their maximum likelihood estimation. In Section~\ref{sec:indep.structures} we investigate the consequences of exchangeability in its interplay with associated Markov properties. In Section~\ref{sec:summarized} we study the consequences of adding the restrictions of summarizedness, and in Section~\ref{sec:consistent} we study consistency properties under subsampling of various systems of network models.

\section{Preliminaries}\label{sec:prelim}
In this section we gather some background material on graph-theoretic concepts,
parametrizations of distributions of binary arrays and graphical modeling.
\subsection{Graph-theoretic concepts}
Following \cite{wes01}, a \emph{labeled graph} is determined by an ordered pair $G = (V,E)$ consisting of a
\emph{vertex} set $V$, an \emph{edge} set $E$, and a relation that with
each edge associates two vertices, called
its \emph{endpoints}. When vertices $u$ and $v$ are the endpoints of an
edge, these are
\emph{adjacent} and we write $u\sim v$; we denote the corresponding edge as
$uv$.

In this paper, we will restrict to \emph{simple} graphs, i.e.\ graphs without
loops (the
endpoints of each edge are distinct) or multiple edges (each pair of vertices are the endpoints
of at most one edge); simple graphs are determined by the pair $G=(V,E)$ alone.  We will distinguish three \emph{types} of edge, denoted by
\emph{arrows}, \emph{arcs} (solid lines with two-headed arrows) and \emph{lines}
(solid lines).  Arrows can be represented by ordered pairs of vertices, while arcs
and lines by $2$-subsets of the vertex set. The graphs in the present paper only contain one of
these types, and are called, respectively, \emph{directed},
\emph{bidirected}, or \emph{undirected}.

The labeled graphs $F=(V_F,E_F)$ and $G=(V_G,E_G)$ are considered equal if and only
if $(V_F,E_F)=(V_G,E_G)$.
We omit the term labeled when the context prevents ambiguity.

A \emph{subgraph} of a graph $G = (V_G,E_G)$ is a graph $F = (V_F,E_F)$ such that
$V_F\subseteq V_G$ and $E_F\subseteq E_G$ and the assignment of endpoints to
edges in $F$
is the same as in $G$. We will write $F \subset G$ to signify that $F$ is a
subgraph of $G$. For a graph $G=(V,E)$, any non-empty subset $A$ of the vertices generates
the subgraph $G(A)$ consisting of all and only vertices in $A$ and edges between two
vertices in $A$,  called the \emph{subgraph induced by} $A$.
Similarly, a subset $B\subseteq E$ of edges induces a subgraph  that contains the edges in $F$ and all and only vertices that are endpoints of edges in $B$.  We note that \emph{edge induced subgraphs have no isolated vertices. }

A \emph{complete graph} is a graph where all nodes are adjacent, and a \emph{$k$-star} is a graph where one node (called the \emph{hub}) is adjacent to all other nodes and there is no other edge in the graph. The \emph{line graph} $L(G)$ of a graph $G=(V,E)$ is the intersection graph of
the edge set $E$, i.e.\ its vertex set is $E$ and $e_1\sim e_2$ if and only if
$e_1$ and $e_2$ have a common endpoint \citep[p.\ 168]{wes01}. We will in
particular be interested in the line graph of the complete graph, which we will
refer to as the \emph{incidence graph}.

\subsection{Homomorphism densities}

We will rely on the notion of subgraph homomorphism density, a graph-theoretic concept that is central to the theory of graph
convergence \citep[see, e.g.,][]{lov12}.

A map $\phi:V_F \to V_G$ between the vertex sets of two  simple
graphs $F=(V_F,E_F)$ and $G=(V_G,E_G)$ is a \emph{homomorphism} if it preserves
edges, i.e.\ if $uv$ in $E_F$ implies $\phi(u)\phi(v)\in E_G$.
A homomorphism $\phi$ is an \emph{isomorphism} if it is a bijection and its inverse is a homomorphism,
i.e.\ $\phi$ also preserves non-edges. If there is an isomorphism between
$F$ and $G$, the graphs are \emph{isomorphic}, and we write $F\iso G$.  An
isomorphism can also be thought of as a \emph{relabelling} of the vertex
set. The isomorphism relation is an equivalence relation among graphs with the
same vertex set. Accordingly, for a graph $G=(V,E)$ we may identify the
corresponding equivalence class as an \emph{unlabeled graph} with $|V|$ vertices. An isomorphism $\phi$
between the graphs $G$ and $G$  is an \emph{automorphism}.  
Then $\aut(G)$  denotes the size of the \emph{automorphism group} ,
i.e.\ the group of automorphisms of $G$.

Following \citet[Ch.\ 5.2]{lov12} and assuming that $|V_F| \leq |V_G|$, we let $\inj(F,G)$ denote
the number of injective graph homomorphisms  between
$F$ and $G$.
Note that if $G$ is a complete graph (i.e. all the vertices are adjacent),  all
maps are homomorphisms, and
$\inj(F,G)=(|V_G|)_{|V_F|}=|V_G|!/(|V_G|-|V_F|)!$. Furthermore, we have that
$\inj(G,G)=\aut(G)$. 
The \emph{injective homomorphism density} of $F$ in $G$ is defined to be
\[
    t_{\mathrm{inj}}(F,G) = \frac{\mathrm{inj}(F,G)}{ (|V_G|)_{|V_F|}},
    \]
    the proportion of all injective mappings from $V_F$ into $V_G$ that are
    homomorphisms. It is immediate from the above definition
    that homomorphism densities are invariant under isomorphisms, in the sense
    that $t_{\mathrm{inj}}(F,G) = t_{\mathrm{inj}}(F',G')$, for any pairs $F
    \iso F'$ and $G \iso G'$.  Thus, density homomorphisms are also well defined over
    unlabeled graphs.
    Indeed, the injective homomorphism density
    $t_{\mathrm{inj}}(F,G)$  can be interpreted as the probability
that $F$ is realized as a random subgraph of $G$ \citep[see,
    e.g.][]{diaconis:janson:08}.

Finally, we will be interested in the number $\sub(F,G)$ of (not necessarily induced) subgraphs of $G$ that  are isomorphic
to $F$.  This is given as
\begin{equation}\label{eq:subcount}\
    \sub(F,G)= \inj(F,G)/\inj(F,F)=\inj(F,G)/\aut(F).
\end{equation}
Clearly, also $\sub(F,G)$ remains unchanged if $F$ and $G$ are replaced by isomorphic
graphs $F'$ and $G'$ so that $\sub(\cdot,\cdot)$ is well defined
over unlabeled graphs.

\subsection{The M\"{o}bius parametrization for binary
arrays}\label{sec:mobius.param}

As explained later in detail, network modeling can be considered equivalent to modeling the joint distribution of a collection of binary random
variables. For this purpose we will find it convenient to rely on a
particular parametrization of such distributions that is based on the  M\"{o}bius
transform \citep[see, e.g.][Appendix A3]{lau96} and that therefore, following
\cite{drt08}, we call the {\it
M\"{o}bius parametrization}.

Let $V$ be a finite set indexing a collection $\{ X_v, v \in V\}$ of binary
random variables taking value in $\{0,1\}^{V}$ and for a non-empty subset $B$ of
$V$ let $X_B = (X_v,v \in B)$.
The M\"{o}bius parametrization arises from taking the
M\"{o}bius transform of the vector of joint probabilities of each point in
$\{0,1\}^V$. This transformation, which is linear and invertible, in turns
yields a vector of marginal, as opposed to joint, probabilities. In detail, let
$\mathcal{B}$ denotes the set of all subsets of $V$ and, for each non-empty $B
\in \mathcal{B}$, set $z_B  = \mathbb{P}(X_B = 1_B)$,
where $1_B$ is the $|B|$-dimensional vector of all $1$'s. Further, set
$z_{\varnothing} = 1$.
Then, using M\"obius inversion, we have that, for each $H \in \mathcal{B}$,
\begin{equation}\label{eq:mobius}
    \mathbb{P}(X_H=1_H,X_{H^c}=0_{H^c})=\sum_{B\in \mathcal{B}:H\subseteq
B}(-1)^{|B\setminus H|}z_B,
\end{equation}
where $H^c = V \setminus H$.
Similarly, for each $B \in \mathcal{B}$,
$$z_B=\mathbb{P}(X_B=1_B)= \sum_{H\in \mathcal{B}:B\subseteq
H} \mathbb{P}(X_H=1_H,X_{H^c}=0_{H^c}).$$
The marginal probabilities $\{ z_B, B \in \mathcal{B}\}$ are called {\it the
    M\"{o}bius parameters\/} of the distribution of $X_V$. The above formulae show that
there is a one-to-one and linear relation between joint probabilities and their
M\"obius parameters and that the only restriction on the set of feasible
M\"obius parameters is that all expressions of the form (\ref{eq:mobius}) should
be non-negative, and $z_\varnothing =1$.

We can represent the set of all strictly positive probabilities on $\{
0,1\}^V$ as an exponential family with canonical sufficient
statistic given by
\[
    x \in \{0,1\}^V \mapsto s(x) \in \{0,1\}^{ |\mathcal{B}| -1 } ,
\]
where $s(x) = \left( s_B(x), B \in \mathcal{B} \setminus \varnothing \right)$, with
\[
s_B(x)=\prod_{b\in B}x_b, \quad B\in\mathcal{B};
\]
see, for instance, \cite{fra86}.
The exponential family on $\{0,1\}^V$ generated by these sufficient statistics
consists of all probability distributions of the form
\begin{equation}\label{model23}
    \mathbb{P}(X = x) = P_{\nu} (x)
=\exp \left\{ \sum_{B \in \mathcal{B} \setminus \varnothing }
\nu_Bs_B(x)-\psi(\nu) \right\}, \quad x \in
\{0,1\}^V,
\end{equation}
for any choice of canonical parameters $\nu = (\nu_B, B \in \mathcal{B} \setminus
\varnothing) \in
\mathbb{R}^{\mathcal{B} \setminus \varnothing}$, and
where $\psi(\nu)$ is the \emph{log-partition function}, ensuring that probabilities add to unity.
The above exponential family is minimal, full and regular, and the  \emph{mean value parameters}
\citep{barndorff:78} for this representation are precisely the M\"{o}bius parameters corresponding to non-empty
subsets of $V$:
\[
    \E\{s_B(X)\} = z_B, \quad B \in \mathcal{B} \setminus \varnothing.
\]

\subsection{Undirected and bidirected graphical models}\label{sec:23}

Graphical models \citep[see, e.g.][]{lau96} are statistical models expressing
conditional independence clauses among a collection of random variables $X_V = (X_v, v
\in V)$ indexed by a finite set $V$. A graphical model is determined by a graph
$G=(V,E)$ over the
indexing set $V$, and the edge set $E$ (which may include edges of undirected, directed
or bidirected type) encodes conditional independence
relations among the variables, or {\it Markov properties}. A joint probability
distribution $P$ for $X_V$ is Markov with respect to $G$ if it satisfies the
Markov properties expressed by $G$.

For non-empty subsets $A$ and $B$ of $V$ and a (possibly empty) subset $S$ of
$V$, we use the notation $A\ci B\cd S$ as shorthand for the conditional
independence relation $X_A\ci X_B\cd X_S$, thus identifying random variables
with their labels. When $S = \emptyset$ the independence relation is
intended as marginal independence.

For an undirected graph $G$ --- where all edges are undirected --- the  (global) Markov property expresses that $A\ci B\cd S$ when every path between $A$ and $B$ has a vertex in $S$ or, in other words, $S$ \emph{separates} $A$ from $B$ in $G$.

We shall be specifically interested in graphical models given by a \emph{bidirected graph} $G$ where all edges are  bidirected. For such graphs the (global) Markov property \citep{cox:wermuth:93,kau96,ric02,ric03}  expresses that

$A\ci B\cd S$ when every path between $A$ and $B$ has a vertex outside $S\cup A\cup B$, i.e.\ $V\setminus(A\cup B\cup S)$ separates $A$ from $B$. Note the obvious duality between this and the Markov property for undirected graphs.

For example, in the undirected graph of Figure \ref{fig:001}(a), 
the global Markov property implies that $\{i,l\}\ci k \cd j$ , whereas in the bidirected graph of Figure \ref{fig:001}(b), 
the global Markov property implies that $\{i,l\}\ci k$. Notice that, for simplicity, we write $k$ and $j$ instead of $\{k\}$ and $\{j\}$ as these are single elements.

\begin{figure}[htb]
\centering
\begin{tikzpicture}[node distance = 5mm and 5mm, minimum width = 3mm]
    \begin{scope}
      \tikzstyle{every node} = [shape = circle,
      font = \scriptsize,
      minimum height = 3mm,
      inner sep = 2pt,
      draw = black,
      fill = white,
      anchor = center],
      text centered]
      \node(i) at (0,0) {$i$};
      \node(j) [right = of i] {$j$};
      \node(l) [above = of j] {$l$};
			\node(k) [right = of j] {$k$};
      \node(i1) [right = 12mm of k] {$i$};
      \node(j1) [right = of i1] {$j$};
      \node(l1) [above = of j1] {$l$};
			\node(k1) [right = of j1] {$k$};
    \end{scope}
		
    \begin{scope}
    \tikzstyle{every node} = [node distance = 5mm and 3mm, minimum width = 5mm,
    font= \scriptsize,
      anchor = center,
      text centered]
\node(a) [below = 2mm  of j]{(a)};
\node(b) [below = 2mm  of j1]{(b)};

\end{scope}
    \begin{scope}
      \draw (i) -- (j);
      \draw (j) -- (l);
			\draw (j) -- (k);
    \end{scope}
    \begin{scope}[<->, > = latex']
    \draw (i1) -- (j1);
      \draw (j1) -- (l1);
			\draw (j1) -- (k1);
    \end{scope}

    \end{tikzpicture}
		\caption{{\footnotesize (a) An undirected dependence graph. (b) a bidirected dependence graph.}}\label{fig:001}
		\end{figure}
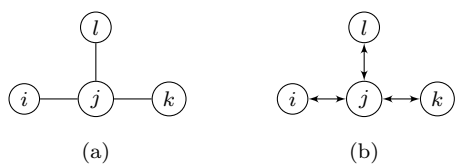

For a system of random variables $X$ we define the \emph{dependence skeleton} of
$X$, denoted by $\skel(X)$ to be the undirected graph with vertex set $V$ such
that vertices $u$ and $v$ are not adjacent if and only if there is \emph{some} subset $S$
of $V$ so that $u\cip v\cd S$.  Thus, if $X$ is Markov with respect to an
undirected graph $G$ --- such as the case considered by \cite{fra86}, who chose $S$ to be the set of remaining variables --- $\skel(X)$ would be a subgraph of $G$; whereas if $X$ is
Markov with respect to a bidirected graph, we may choose $S=\emptyset$ and the corresponding skeleton will be a
subgraph of the graph obtained by replacing all arcs
with lines.

\paragraph{Bidirected graphical models for binary variables} For the bidirected case, the exponential representation does not simplify. However, in the M\"obius parametrization the bidirected Markov property takes a particular simple form. Indeed \cite{drt08} showed that a distribution $P$ is Markov with respect to the bidirected graph $G=(V,E)$ if and only if
for any $B \subset V$ that is disconnected in $G$, the corresponding M\"obius parameter satisfies
\begin{equation}\label{eq:prod_rest}z_B=z_{C_1}\cdots z_{C_k},\end{equation}
where $C_1,\ldots,C_k$  partitions $B$ into inclusion maximal connected sets
(with respect to $G$).
Thus, if we let $\mathcal{C} = \mathcal{C}(G)$ be the set of all non-empty
connected subsets of $V$ (with respect to $G$),
the bidirected binary graphical model with graph $G$ is completely, injectively,
and smoothly parametrized by the M\"obius parameters $(z_C, C\in\mathcal{C})$.
More precisely we have that
\begin{equation}\label{model21}
    \mathbb{P}(X_H=1_H,X_{H^c}=0_{H^c})=\sum_{B\in \mathcal{B}:H\subseteq B}(-1)^{|B\setminus H|}\prod_{C\in \mathcal{C}_B}z_C,
\end{equation}
where $\mathcal{C}_B$ denotes the collection of inclusion maximal connected subsets of $B$.
Further, the M\"obius parameters vary freely save for the restrictions that
$z_\varnothing =1$ and the right-hand expressions in (\ref{eq:mobius}) and
(\ref{model21}) must be non-negative for all choices of $E\subseteq V$.
For example, for the bidirected graph of Figure \ref{fig:20} we have that
\begin{equation*}
    \mathbb{P}(X_1=1,X_2=0,X_3=0)=z_1-z_{12}-z_1z_3+z_{123},\quad \mathbb{P}(X_1=1,X_3=1)= z_1z_3.
\end{equation*}

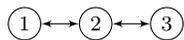
\begin{figure}[htb]
\centering
\begin{tikzpicture}[node distance = 5mm and 5mm, minimum width = 3mm]
    \begin{scope}
      \tikzstyle{every node} = [shape = circle,
      font = \scriptsize,
      minimum height = 3mm,
      inner sep = 2pt,
      draw = black,
      fill = white,
      anchor = center],
      text centered]
      \node(a) at (0,0) {$1$};
      \node(b) [right = of a] {$2$};
      \node(c) [right = of b] {$3$};
    \end{scope}

    \begin{scope}[<->, > = latex']
      \draw (a) -- (b);
      \draw (b) -- (c);

    \end{scope}

    \end{tikzpicture}
\caption{{\footnotesize A dependence graph with vertex set $\{1,2,3\}$.}}\label{fig:20}
\end{figure}

\cite{drt08} further show that the family of positive bidirected Markov distributions becomes {\it a curved
exponential family\/} due to the restriction (\ref{eq:prod_rest}) which is
non-linear  in the canonical parameters $\nu_B$ and log-linear in the mean value
parameters $z_B$ \citep{roverato:etal:13}.
Finally,  the dimension of the corresponding model is the cardinality of $\mathcal{C}$, the
number of connected induced subgraphs of $G$.

\section{Network models and exchangeability}\label{sec:5}

\subsection{Some terminology}
Given a finite or countably infinite node set $\nodeset$ ---
representing \emph{individuals} or \emph{actors} in a given population of
interest ---  we define a
\emph{random network} over $\nodeset$ to be a collection $X = (X_d, d\in
D(\nodeset) )$
of binary random variables taking values $0$ and $1$ indexed by a set
$D(\nodeset)$ of \emph{dyads}. Throughout the paper we take
$D(\nodeset)$ to be the set of all
unordered pairs $ij$ of nodes in $\nodeset$, and
nodes $i$ and $j$ are said to have
a \emph{tie}  if the random variable $X_{ij}$ takes the value $1$, and no tie
otherwise. 
Thus, a network is a random
variable taking value in $\{0,1 \}^{ { \nodeset \choose 2}  }$ and can therefore
be seen as a random {\it simple, undirected, labeled} graph with node set $\nodeset$,
whereby the ties form the random edges of the graphs.
We will write $\mathcal{G}_{\nodeset}$ for the set of all simple,
labeled undirected graphs on
$\nodeset$.

We use the terms
network, node, and tie rather than (random) graph, vertex, and (random) edge to differentiate
from the terminology used in the graphical model sense. Indeed, as we shall
discuss graphical models for networks, we will also consider each dyad $d$ as a
vertex in a graph $G=(D,E)$ representing the dependence structure of the random
variables associated with the dyads. 

\subsection{Exponential random graph models}

\emph{Exponential random graph models} \citep{frank:91,wasserman:pattison:96},
or ERGMs in short, are among the most important and popular statistical models
for network data and, as we will see, the set of exchangeable distributions form an ERGM.

ERGMs are exponential families of distributions \citep{barndorff:78} on
$\mathcal{G}_{\nodeset}$  whereby the probability of observing a network $x \in
\mathcal{G}_{\nodeset}$ can
be expressed as
\begin{equation}\label{eq:ergm}
    P_{\theta}(x) = \exp\Big\{\sum_{l=1}^ms_l(x)\theta_l-\psi(\theta)\Big\},
    \quad \theta \in \Theta \subseteq \mathbb{R}^m,
\end{equation}
where $s(x) = (s_1(x) \ldots, s_m(x)) \in \mathcal{R}^m$ are \emph{canonical sufficient
statistics} which capture some important feature of $x$, $\theta \in
\mathbb{R}^m$ is a point in the canonical parameter space $\Theta$,  and
$\psi \colon \Theta \rightarrow [0,\infty)$ is the log-partition function; in (\ref{eq:ergm}) we have thus chosen counting measure as the base measure of the representation.

The choice of the canonical sufficient statistics and any restriction imposed on
$\Theta$ will determine the properties of the corresponding ERGM. The simplest ERGM is the \emph{Erd\"{o}s--R\'{e}nyi} model \citep{erdos:renyi:60}, where there is only one parameter $\theta$, and the canonical sufficient statistic is the number of ties in $x$. This is equivalent to ties occurring independently with probability $p$ and $\theta=\log\{{p}/{(1-p)}\}$.

\emph{Beta models} are also ERGMs and they can be considered a generalization of the Erd\"{o}s-R\'{e}nyi model: in these models it is also assumed that ties occur independently; the probability $p_{ij}$ of a tie between nodes $i$ and $j$ is given as:
\begin{equation}\label{eq:beta}
    p_{ij}=\mathbb{P}(X_{ij}=1)=\frac{e^{\beta_i+\beta_j}}{1+e^{\beta_i+\beta_j}},\nn \forall ij\in D,
\end{equation}
where  $(\beta_i, i\in\nodeset)$ can be interpreted as  parameters that
determine the propensity of node $i$ to have edges. The probability  of a
network $x \in \mathcal{G}_{\nodeset}$ is thus
\begin{equation}\label{eq:betarep}
P_{\beta}(x)=\prod_{ij\in D}p_{ij}^{x_{ij}}(1-p_{ij})^{1-x_{ij}}=\prod_{ij\in D}\frac{e^{(\beta_i+\beta_j)x_{ij}}}{1+e^{(\beta_i+\beta_j)}}=
\exp\left\{\sum_{i\in\nodeset}d_i(x)\beta_i-\psi(\beta)\right\},
\end{equation}
where $(d_i(x), i\in \nodeset)$ is  the \emph{degree sequence} of $x$, i.e.\ $d_i(x)$ is the number of ties in $x$ involving node $i$.

Other ERGMs include, for example, the exponential family models of \cite{hol81} and their modifications  \citep{fie81}, as well as the Markov graphs of \cite{fra86}.

If $\nodeset'\subseteq \nodeset$, any probability distribution $P_\nodeset$ for
a network $X_{\nodeset} = (X_d, d \in D(\nodeset))$
 induces a distribution $P_{\nodeset'}$
 for the subnetwork $ X_{\nodeset'} = (X_d, d \in D(\nodeset'))$
 corresponding to the dyads in $D(\nodeset')$. In particular, $X_{\nodeset'}$ is the
 subgraph of $X_{\nodeset}$ induced by $\nodeset'$.
 One problem with ERGMs is that, typically, $P_{\nodeset'}$ needs not be the
 related to  $P_{\nodeset}$ in a meaningful way, and needs not be an ERGM itself. In that sense, the ERGMs are
 not  \emph{marginalizable}; see also \cite{snijders:10} as well as
 \cite{shalizi:rinaldo:13}.  The Erd\"os--R\'{e}nyi models and the beta models
 are marginalizable in this sense. Many other network models suffer
 from this issue: see, e.g., \cite{CD:15}. We shall  return to this and related issues in more detail in Section~\ref{sec:consistent}.

\subsection{Parametrizations of random network
models}\label{sec:network.parametrizations}

Since networks are arrays of binary variables, we may use the M\"{o}bius
parametrization and the exponential form  in \Cref{sec:mobius.param} to represent their distributions.

For a given finite node set $\nodeset$, we thus let
$\mathcal{B}(\nodeset)  = \mathcal{B}$ be the collection of all subsets of
$D(\nodeset)$. Now if $X$ is a network on $\nodeset$ and $B$ a non-empty set in
$\mathcal{B}$, then  $X_B = (X_d, d \in B)$ is the subnetwork
indexed by the
dyads in $B$. It will be convenient to identify each non-empty set $B \in \mathcal{B}$
with the  edge-induced subgraph of the complete graph on
$\nodeset$ comprised of the edges in $B$. Equivalently, $\mathcal{B} \setminus
\varnothing$ represents all subgraphs of the complete graph on
$\nodeset$ without isolated nodes. In particular, for a non-empty $B \in \mathcal{B}$,
we may equivalently  write the event $\{ X_B = 1_B \}$ as the event $\{ B
\subset X \}$ that the graph $B$ is
a subgraph of the random graph $X$.

The M\"{o}bius parameters corresponding to the distribution $P$ of a network $X$
are $(z_B, B \in \mathcal{B})$, where
\[
    z_B = \left\{
    \begin{array}{ll}
 \mathbb{P}( B \subseteq X) & \text{if } B \neq
	\varnothing\\
	1 & \text{otherwise.}\\
    \end{array}
    \right.
\]
 Thus, the M\"{o}bius parameters for a network
distribution are just the implied probabilities of observing all subgraphs of
the complete graph on $\nodeset$
without isolated nodes.
In the exponential family parametrization \eqref{model23}, the canonical
sufficient statistic $s$ is defined over graphs $\mathcal{G}_{\nodeset}$, and
the coordinate indexed by any non-empty $B \in \mathcal{B}$
 is simply the indicator function
\[s_B(x) = \left\{\begin{array}{ll}1&\mbox{if $B\subseteq x$}\\0&\mbox{otherwise}.\end{array}\right.\]

\begin{example}\label{ex:1}
Suppose that we observe the network $x$ in Figure \ref{fig:3} (a), where
$\nodeset = \{1,2,3,4\}$.

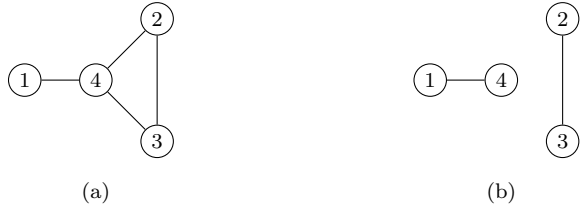
\begin{figure}[htb]
\centering
\begin{tikzpicture}[node distance = 5mm and 5mm, minimum width = 3mm]
    \begin{scope}
      \tikzstyle{every node} = [shape = circle,
      font = \scriptsize,
      minimum height = 3mm,
      inner sep = 2pt,
      draw = black,
      fill = white,
      anchor = center],
      text centered]
      \node(a) at (0,0) {$1$};
      \node(b) [right = of a] {$4$};
      \node(c) [above right= of b] {$2$};
			\node(d) [below right = of b] {$3$};
			
			\node(a1) [right = 40mm of b]{$1$};
      \node(b1) [right = of a1] {$4$};
      \node(c1) [above right= of b1] {$2$};
			\node(d1) [below right = of b1] {$3$};

    \end{scope}
		
    \begin{scope}
    \tikzstyle{every node} = [node distance = 5mm and 5mm, minimum width = 3mm,
    font= \scriptsize,
      anchor = center,
      text centered]
\node(a2) [below = 10mm  of b]{(a)};
\node(b2) [below = 10mm  of b1]{(b)};

\end{scope}
    \begin{scope}
      \draw (a) -- (b);
      \draw (b) -- (c);
			\draw (b) -- (d);
			\draw (c) -- (d);
			\draw (a1) -- (b1);
			\draw (c1) -- (d1);
    \end{scope}

    \end{tikzpicture}
  
  \caption[]{\small{Two observed networks with four nodes.}}
     \label{fig:3}
     
\end{figure}
For any
probability measure $P$ on $\mathcal{G}_\nodeset$, we have that, under both the M\"obius
parametrization and the canonical parametrization in (\ref{model23}),
\begin{equation}\label{eq:110}
\begin{split}
   P_{\nu}(x)=z_{14,23,24,34}-z_{12,14,23,24,34}-z_{13,14,23,24,34}+z_{12,13,14,23,24,34}\\= \exp\{\nu_{14}+\nu_{23}+\nu_{24}+\nu_{34}+\nu_{14,23}+\nu_{14,24}+\nu_{14,34}+\nu_{23,24}+\nu_{23,34}+\\\nu_{24,34}+\nu_{14,24,23}+\nu_{14,24,34}
+\nu_{14,34,23}+\nu_{13,14,34}+\nu_{14,23,24,34}-\psi(\nu)\},
\end{split}
\end{equation}
where $\nu$ is the unique $63$-dimensional vector parametrizing the distribution
of $X$.
Notice that the M\"obius parametrization is considerably simpler than the
exponential family parametrization in this  instance, since $x$ is a dense network and only parameters corresponding to supergraphs of the observed network enter into the calculation.
However, if we instead observe the sparser network $y \in \mathcal{G}_{\nodeset}$ in Figure~\ref{fig:3} (b), we get
\begin{equation}\label{eq:sparse}
\begin{split}
    P_{\nu}(y)=z_{14,23} -   z_{14,23,12}-z_{14,23,13}-z_{14,23,24}-z_{14,23,34}\\
+z_{14,23,12,13} +z_{14,23,12,24}+z_{14,23,12,34}+z_{14,23,13,34}+z_{14,23,13,24}+z_{14,23,24,34}\\-z_{12,14,23,24,34}-z_{13,14,23,24,34}+z_{12,13,14,23,24,34}\\=\exp\{\nu_{14}+\nu_{23}+\nu_{14,23}-\psi(\nu)\}.
\end{split}
\end{equation}

Thus the exponential form is simpler for the sparse case, save for the fact that
the log-partition function $\psi(\nu)$ is complicated.
\end{example}

\subsection{Exchangeability}

We are concerned with probability distributions on networks that are
{\it exchangeable:}  invariant under relabelings of the node set $\nodeset$.
Exchangeability is a well-known form of probabilistic invariance, and, from a modeling
perspective, also a natural simplifying assumption to impose when formaliziing
statistical models for networks. Below we summarize
well-known facts on exchangeability.

A distribution $P$ of a random  array $X=(X_{ij})_{i,j\in \nodeset}$  over a finite nodeset $\nodeset$ is said to be \emph{weakly exchangeable} (WE) \citep{silverman:76,eagleson:weber:78} if  for all permutations $\pi\in S(\nodeset)$ we have that
\begin{equation}
    \mathbb{P}\{(X_{ij}=x_{ij})_{i,j\in \nodeset}\}=\mathbb{P}\{(X_{ij}=x_{\pi(i)\pi(j)})_{i,j\in \nodeset}\}.
\end{equation}
If the array $X$ is symmetric --- i.e.\ $X_{ij}=X_{ji}$, we say it is \emph{symmetric weakly exchangeable} (SWE). In the following we shall for brevity say that $X$ is WE, SWE, etc.\ in the meaning that its distribution is.

A symmetric binary array with zero diagonal can be interpreted as a matrix of ties (adjacency matrix) of a random network and thus the above concepts can be translated into networks. A random network is \emph{exchangeable} if its adjacency matrix is  SWE.  Then
 a random network is exchangeable if and only if its distribution is invariant
 under relabeling of the nodes of the network. We will distinguish the case of a
 finite node set  from that of a (countably) infinite one, and in the former case we will speak of {\it finite
 exchangeability}.
A random network $X$ over a countably infinite nodeset $\nodeset$ is said to be exchangeable,
if every finite induced subnetwork $X_{\nodeset'}$ for $\nodeset'\subset \nodeset$ is.

When $\nodeset$ is infinite,  exchangeability is well-understood.
Indeed, any doubly infinite random SWE array can be represented as a mixture of
so-called \emph{$\phi$-matrices} \citep{diaconis:freedman:81} in the following
way. 
 Let $\mathcal{P}_{\infty}$ be the convex set of all exchangeable distributions for networks with a countably infinite number of nodes and let $\mathcal{E}_{\infty}$ denote the extreme points of that set. Then we have:
\begin{prop}\label{prop:1n} $P\in \mathcal{E}_{\infty}$ if and only if for every
    finite $\nodeset' \subset \nodeset$,
\begin{equation}\label{eq:21}
    P_{\nodeset'}(x)=\int_{[0,1]^\nodeset}\prod_{ij\in
    D(\nodeset')}\phi(u_i,u_j)^{x_{ij}}\{1-\phi(u_i,u_j)\}^{1-{x_{ij}}}\,du,
    \quad  \forall x \in \mathcal{G}_{\nodeset'},
\end{equation}
where $P_{\nodeset'}$ is the induced probability on $\mathcal{G}_{\nodeset'}$
and $\phi$ is a measurable function from $[0,1]^2$ to $[0,1]$. The function $\phi$ is unique up to a measure-preserving transformation of the unit interval.
\end{prop}
In a graph theoretic context, (equivalence classes of) $\phi$-matrices are also known as \emph{graphons} \citep{lov06}.

Elements of $\mathcal{E}_{\infty}$ as given in (\ref{eq:21}) are all
\emph{dissociated} \citep{silverman:76} in the sense that  $\{X_{ij}, i,j\in
\nodeset'\}$ are independent of $\{X_{kl}, k,l\in \nodeset''\}$ whenever
$\nodeset'\cap\nodeset''=\emptyset$ in fact it holds \citep{aldous:81,aldous:85} that
\begin{prop}
$P\in \mathcal{E}_{\infty}$  if and only if $P$ is exchangeable and dissociated.
\end{prop}

We emphasize that the representation of exchangeable distributions on networks by means of
$\phi$-matrices requires an infinite node set $\nodeset$. If $\nodeset$ is
finite, a finitely exchangeable network on $\nodeset$ needs not have a
representation as a mixture of $\phi$-matrices,
and need not be {\it extendable}, i.e.\ the induced probability measure from a
probability distribution over networks on a larger node set. In fact, the properties of
finitely exchangeable networks are distinctively different; see for example \cite{dia77,diaconis:freedman:80,Mat95,ker06}; and \cite{KL:15}.
We will return to  the issue of extendability later in \Cref{sec:consistent} and
focus instead on the properties of finitely exchangeable network distributions.

\subsection{Exponential representation of exchangeable networks}\label{sec:exch.parametrization}

For exchangeable random networks on a finite  nodeset $\nodeset$, both the M\"obius and the
exponential parameters simplify. More precisely, $P$ is an exchangeable
distribution on $\mathcal{G}_{\nodeset}$ if and only if $P(x) =
P(x')$ whenever $x\iso x'$, i.e.\ whenever $x$ and $x'$ are identical save
for a permutation of their labels.  Thus,
it follows that a random network $X$ is exchangeable if and only if $z_B= z_{B'}$
whenever $B\iso B'$, where $z_B= \mathbb{P}(B\subseteq X)$.

Next we let $\mathcal{U}_\nodeset$ denote the set of all unlabeled graphs
 on $\nodeset$ and write $\varnothing$ for the empty graph. Each $U \in
 \mathcal{U}_{\nodeset}$ can be identified with an  isomorphism class in
 $\mathcal{G}_{\nodeset}$,  which we will denote with $[U]$. 
Then, any exchangeable distribution $P$ on
$\mathcal{G}_{\nodeset}$ is parametrized by $\{ z_U, U\in
    \mathcal{U}_\nodeset\}$, with $z_{\varnothing} = 1$, where $z_U =
    \mathbb{P}\left( B \subseteq X \right)$, for any $B \in [U]$. Collecting identical terms in (\ref{eq:mobius}) we obtain
\begin{equation}\label{eq:mobexch}
    \mathbb{P}(X = x)=\sum_{U\in \mathcal{U}_\nodeset}(-1)^{|U|-|x|}r_U(x)z_U,
    \quad x \in \mathcal{G}_{\nodeset}
\end{equation}
where $|U|$ is the number of edges in $U$ and $r_U(x)$ is the number of graphs
in $[U]$ that contain $x$ as a subgraph. 

Similarly, it holds for the exponential representation \eqref{model23} that a
probability distribution on $\mathcal{G}_{\nodeset}$ is exchangeable if and only
if for all $B$ and $B=B'$ in $D(\nodeset)$ with $B\iso B'$, the corresponding
canonical parameters satisfy $\nu_B=\nu_{B'}$. Thus we can represent the family
of exchangeable network distributions on $\nodeset$ as the exponential family of
probability distributions of the form
\begin{equation}\label{eq:expexch}
    P_{\nu}(x)=\exp \Big\{\sum_{U\in\mathcal{U}_\nodeset \setminus
\varnothing}\sigma_U(x)\nu_U-\psi(\nu) \Big\}, \quad x \in \mathcal{G}_{\nodeset},
\end{equation}
for any choice of the canonical parameters $\nu = \left( \nu_U, U \in
\mathcal{U}_{\nodeset} \setminus \varnothing \right) \in \mathbb{R}^{
    \mathcal{U}_{\nodeset} \setminus \varnothing}$,
    where $\psi(\cdot)$ is  the log-partition function and $$\sigma_U(x)= \sum_{B \in [U]}s_{B}(x)$$ is the number of graphs
    in the isomorphism class corresponding to $U$  that are subgraphs of $x$. Indeed from (\ref{eq:subcount}) we have that
\begin{equation}\label{eq:suffbyinj}\sigma_U(x)=\inj(U,x)/\inj(U,U).
\end{equation}

Note that the set of exchangeable distributions again form a linear and regular
exponential family with canonical sufficient statistics
$(\sigma_U, U \in \mathcal{U}_{\nodeset} \setminus \varnothing)$,
canonical parameters $(\nu_U, U\in \mathcal{U}_\nodeset \setminus \varnothing)$, and mean value parameters
\begin{equation}\label{eq:meanvalue}\tau_U= \E\{\sigma_{U}(X)\}=  \sum_{B \in
    [U]}z_U=\sub(U,D)z_U=\inj(U,D)z_U/\inj(U,U) \end{equation}
for $U \in
    \mathcal{U}_{\nodeset} \setminus \varnothing$. In other words, \emph{the family of finitely exchangeable network distributions is an ERGM}.

By standard theory of exponential families \citep{barndorff:78}, upon observing
a network $x \in \mathcal{G}_{\nodeset}$, the maximum likelihood estimate of the M\"{o}bius parameters $\left(
z_U, U\in \mathcal{U}_\nodeset \setminus \varnothing \right)$ under the assumption
of exchangeability is obtained by equating the observed canonical statistic to
its expectation. As a result, the MLE of $z_U$, for any non-empty $U$ is
\begin{equation}\label{eq:zmle}\hat z_U= \sigma_U(x)/\sub(U,D)=\inj(U,x)/\inj(U,D)=t_{\mathrm{inj}}(U,x).\end{equation}
Strictly speaking,  this is not the maximum likelihood estimate in the
exponential family, but rather its \emph{completion} \citep{barndorff:78}
obtained by including all limit points of the set of mean value parameters. In
fact, the canonical parameters $\nu_U$ are not estimable based on  observation of a single  network.
Note that, interestingly, \emph{the maximum likelihood estimates of the M\"{o}bius parameters  are exactly the injective homomorphism densities introduced in Section
\ref{sec:prelim}.} For an extreme and infinitely exchangeable network (the dissociated case), these estimates are also well known to be consistent  \cite[Theorem 2.5]{lov06}.

\begin{example}\label{ex:exch}
Consider the networks $x$ and $y$ in Figure~\ref{fig:3}.
Using the exponential representation (\ref{eq:expexch}) we see --- with a notation that hopefully is transparent
--- that the probability of observing $x$ is
\begin{equation}\label{eq:11}
P_\nu(x) = \exp\left\{4\nu_{\n{\vcenter{\hbox{\includegraphics[scale=0.05]{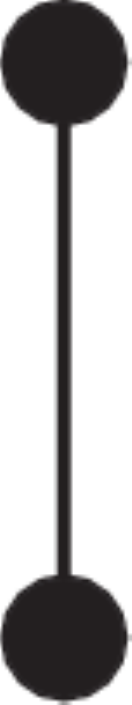}}}}}+5\nu_{\n{\vcenter{\hbox{\includegraphics[scale=0.05]{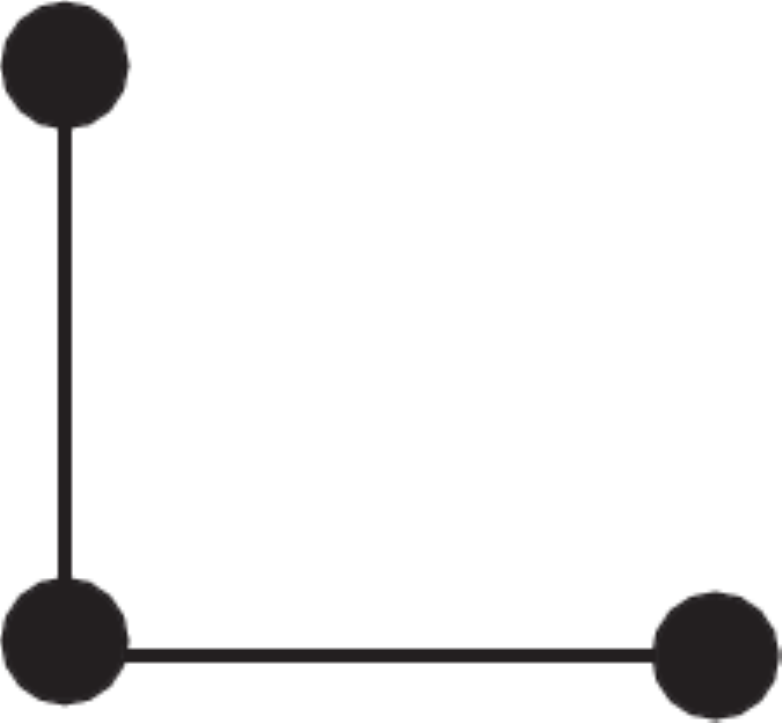}}}}}+\nu_{\n{\vcenter{\hbox{\includegraphics[scale=0.05]{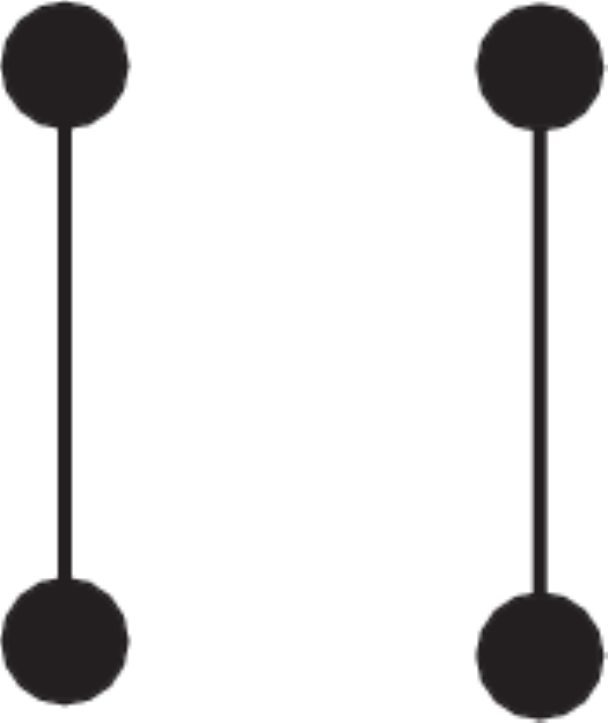}}}}}+\nu_{\n{\vcenter{\hbox{\includegraphics[scale=0.05]{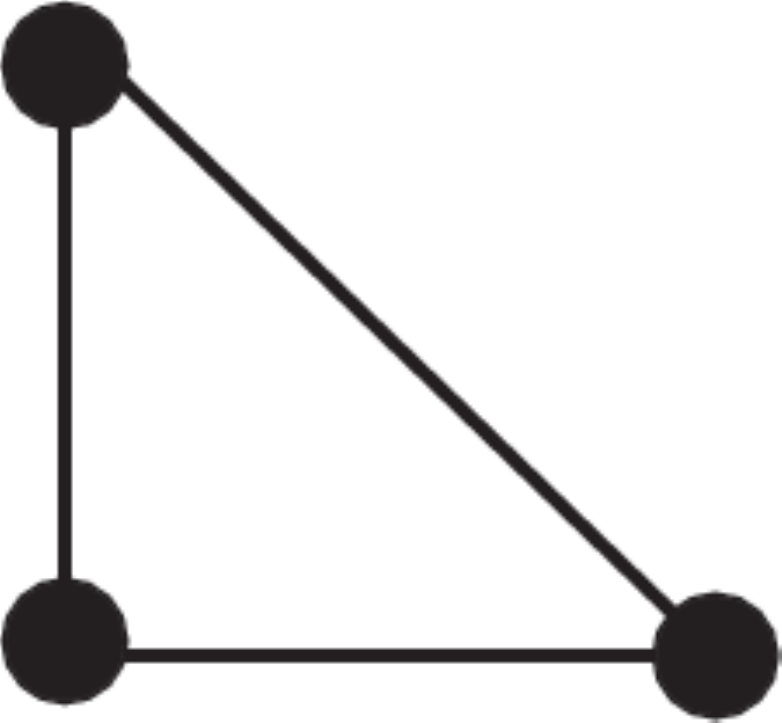}}}}}+\nu_{\n{\vcenter{\hbox{\includegraphics[scale=0.05]{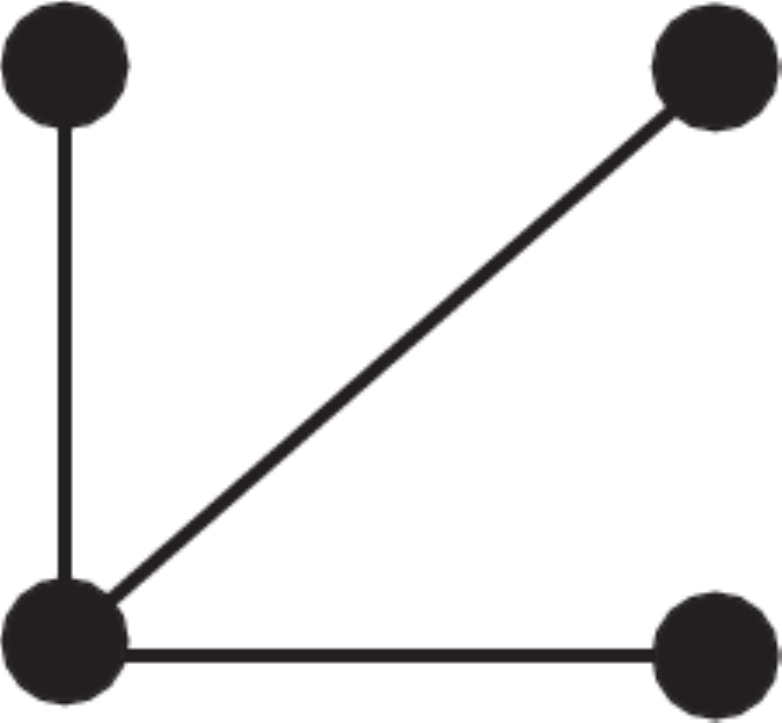}}}}}+2\nu_{\n{\vcenter{\hbox{\includegraphics[scale=0.05]{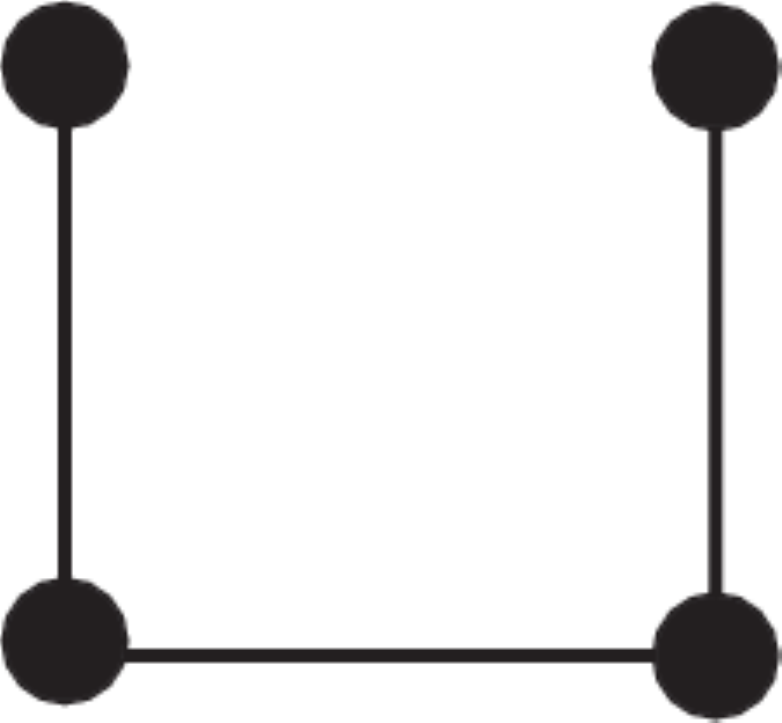}}}}}+\nu_{\n{\vcenter{\hbox{\includegraphics[scale=0.05]{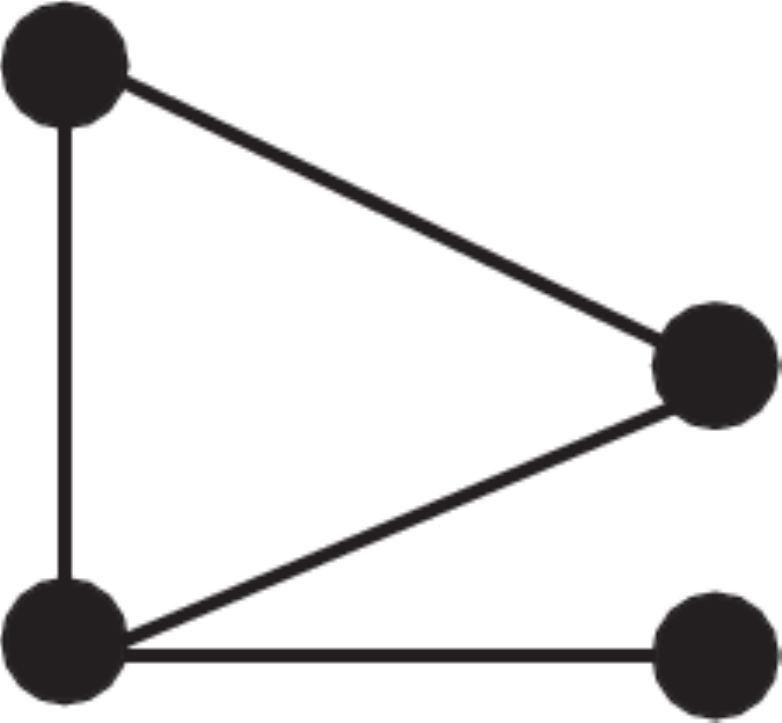}}}}}-\psi(\nu)\right\}.
\end{equation}

In terms of the M\"obius parametrization the expression simplifies to
\begin{equation}\label{eq:xmobexch}
    P_{\nu}(x) =z_{\n{\vcenter{\hbox{\includegraphics[scale=0.05]{figureex2-7.pdf}}}}}-2z_{\n{\vcenter{\hbox{\includegraphics[scale=0.05]{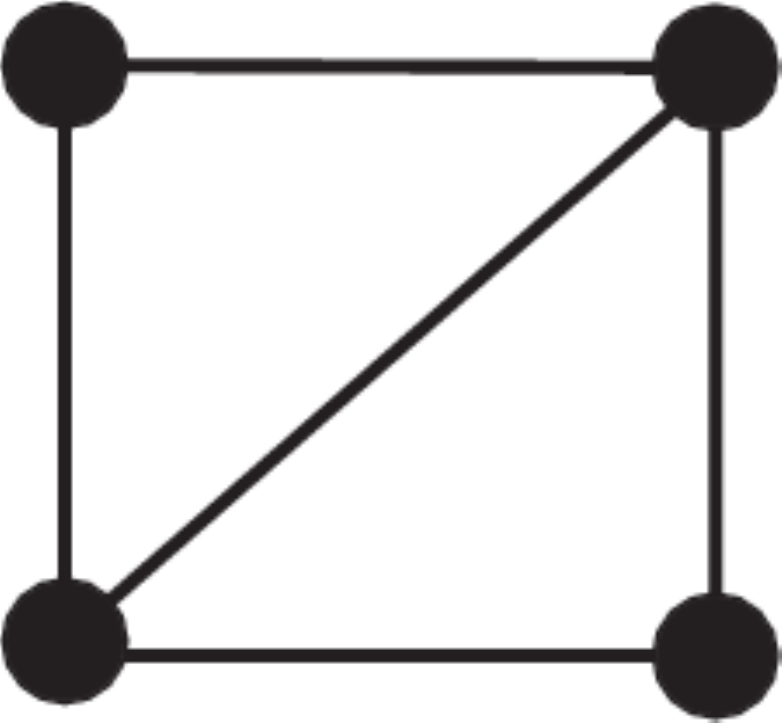}}}}}+z_{\n{\vcenter{\hbox{\includegraphics[scale=0.05]{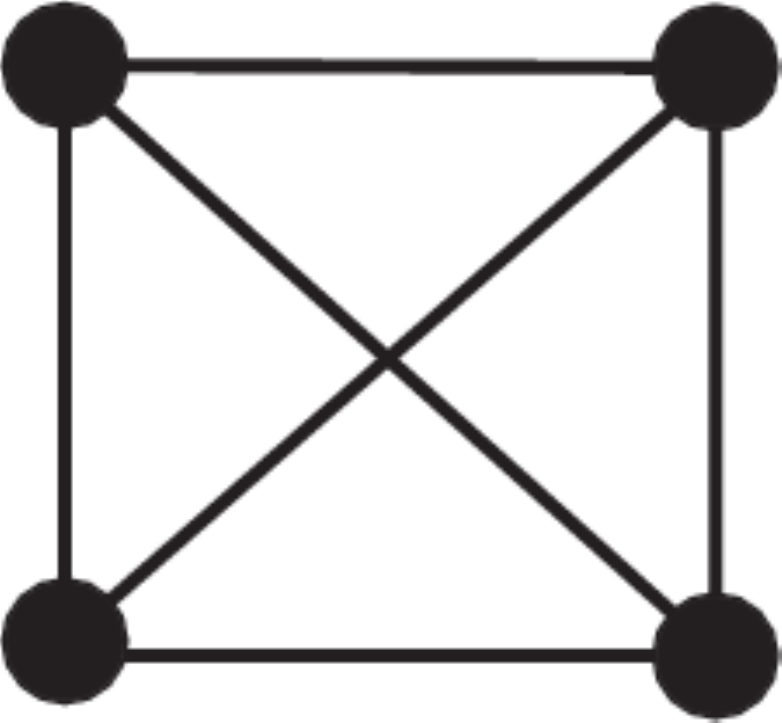}}}}}.\end{equation}
If $x$ is observed, the non-zero estimates of M\"obius parameters for $U\neq
\varnothing$ are
\begin{equation}\label{eq:xexchmle}\begin{split}\hat z_{\n{\vcenter{\hbox{\includegraphics[scale=0.05]{figureex2-1.pdf}}}}}=2/3,\; \hat z_{\n{\vcenter{\hbox{\includegraphics[scale=0.05]{figureex2-2.pdf}}}}}=5/12,\;\hat z_{\n{\vcenter{\hbox{\includegraphics[scale=0.05]{figureex2-3.pdf}}}}}=1/3,\;\hat z_{\n{\vcenter{\hbox{\includegraphics[scale=0.05]{figureex2-4.pdf}}}}}=1/4,\\ \hat z_{\n{\vcenter{\hbox{\includegraphics[scale=0.05]{figureex2-5.pdf}}}}}=1/4,\;\hat z_{\n{\vcenter{\hbox{\includegraphics[scale=0.05]{figureex2-6.pdf}}}}}=1/6,\;\hat z_{\n{\vcenter{\hbox{\includegraphics[scale=0.05]{figureex2-7.pdf}}}}}=1/12;\end{split}\end{equation}
These numbers are obtained from (\ref{eq:zmle}) by calculating, for example, as $$\inj(\n{\vcenter{\hbox{\includegraphics[scale=0.05]{figureex2-4.pdf}}}},x)= 3!=6, \quad\inj(\n{\vcenter{\hbox{\includegraphics[scale=0.05]{figureex2-4.pdf}}}}, D)= 4\times 3\times 2 =24.$$
Note that the canonical parameters in (\ref{eq:11})
are not
estimable in this case as the set of sufficient statistics $(\sigma_U(X), U\in
\mathcal{U}_\nodeset \setminus \varnothing) $ are  at the  boundary of their convex support.
\end{example}

\section{Independence structures for exchangeable
networks}\label{sec:indep.structures}
In this section we will investigate the Markov properties of exchangeable
network distributions on a finite node set.
\subsection{Possible independence structures}
Recall that we use the term \emph{incidence graph} $L(\nodeset)$ for the line
graph of a complete graph on $\nodeset$, whereby $L(\nodeset)$ has edges between
dyads which are \emph{incident} i.e.\ dyads having a node in common.
Figure~\ref{fig:2} displays the incidence undirected graph  for $\nodeset = \{1,2,3,4\}$ and a bidirected version $L_{\leftrightarrow}(\nodeset)$ of the same.
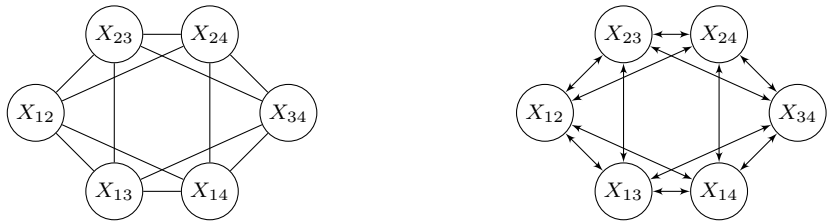
\begin{figure}[htb]
\centering
\begin{tikzpicture}[node distance = 5mm and 5mm, minimum width = 3mm]
    \begin{scope}
      \tikzstyle{every node} = [shape = circle,
      font = \scriptsize,
      minimum height = 3mm,
      inner sep = 2pt,
      draw = black,
      fill = white,
      anchor = center],
      text centered]
      \node(12) at (0,0) {$X_{12}$};
      \node(13) [below right = of 12] {$X_{13}$};
      \node(23) [above right = of 12] {$X_{23}$};
			\node(24) [right = of 23] {$X_{24}$};
      \node(14) [right = of 13] {$X_{14}$};
      \node(34) [below right = of 24] {$X_{34}$};
			
			\node(a12) [right = 6cm of 12] {$X_{12}$};
      \node(a13) [below right = of a12] {$X_{13}$};
      \node(a23) [above right = of a12] {$X_{23}$};
			\node(a24) [right = of a23] {$X_{24}$};
      \node(a14) [right = of a13] {$X_{14}$};
      \node(a34) [below right = of a24] {$X_{34}$};
    \end{scope}

    \begin{scope}
      \draw (12) -- (13);
			\draw (12) -- (14);
			\draw (12) -- (23);
			\draw (12) -- (24);
      \draw (13) -- (23);
			\draw (13) -- (14);
			\draw (13) -- (34);
			\draw (23) -- (24);
			\draw (23) -- (34);
			\draw (14) -- (24);
			\draw (24) -- (34);\draw (14) -- (34);

    \end{scope}

    \begin{scope}[<->, > = latex']
      \draw (a12) -- (a13);
			\draw (a12) -- (a14);
			\draw (a12) -- (a23);
			\draw (a12) -- (a24);
      \draw (a13) -- (a23);
			\draw (a13) -- (a14);
			\draw (a13) -- (a34);
			\draw (a23) -- (a24);
			\draw (a23) -- (a34);
			\draw (a14) -- (a24);
			\draw (a24) -- (a34);\draw (a14) -- (a34);

    \end{scope}

    \end{tikzpicture}
		\caption{\small{(a) The incidence graph for $\nodeset
		    \{1,2,3,4\}$. (b) The bidirected incidence graph with $4$ nodes.}}
     \label{fig:2}
		\end{figure}
 In fact, we show next that the skeleton graph of a finitely
exchangeable network distribution can only be one of four possible types.

\begin{prop}
    \label{prop:skeleton}
The dependence skeleton $\skel(X)$ of an exchangeable random network $X$ is one of the following:
\begin{enumerate}[\rm(a)]
  \item the empty graph;
  \item the incidence graph $L(\nodeset)$ ;
  \item the complement of the incidence graph $L^c(\nodeset)$ ;
  \item the complete graph.
\end{enumerate}
\end{prop}
\begin{proof}
    Let $| \nodeset | = n$.
Exchangeability implies that one can relabel the ${n \choose 2}$ vertices of the dependence skeleton corresponding to every permutation of  $(1,\dots,n)$  without changing independence statements of form $i\ci j\cd S$ for some $S\subset D$. If the dependence skeleton is empty or complete it is clear that relabeling would not change such independence statements.

If there is an edge $ij,kl$, for mutually non-equal $i,j,k,l$, and a missing edge $i'j',k'l'$, for mutually non-equal $i',j',k',l'$, in the dependence skeleton, then by the relabeling corresponding to sending $(i',j',k',l')$ to $(i,j,k,l)$  we obtain a missing edge between  $ij$ and $kl$. However, in the latter labeling, there is an independence of form $ij\ci kl \cd S$ whereas in the former labeling there is no independence of form $ij\ci kl \cd S$. Hence, such a graph cannot be the dependence skeleton of an exchangeable random network.

Similarly, if there is an edge $ij,il$, for mutually non-equal $i,j,l$, and a missing edge  $i'j',i'l'$, for mutually non-equal $i',j',l'$, in the dependence skeleton, then by the relabeling corresponding to sending $(i',j',l')$ to $(i,j,l)$ we obtain a missing edge between  $ij$ and $il$. With the same argument as in the previous case, such a graph cannot be the dependence skeleton of an exchangeable random network.

The remaining cases are when either all the pairs of vertices $ij,il$ are adjacent and all the pairs of vertices $ij,kl$ are not adjacent, which leads to the $L(\nodeset)$ ; or all the pairs of vertices $ij,kl$ are adjacent and all the pairs of vertices $ij,il$ are not adjacent, which leads to $L^c(\nodeset)$ .
\halm \end{proof}

To give a complete characterization of all possible graphical Markov structures with these skeletons, we need to give a precise definition of a generic graphical Markov structure. We abstain from doing so but just mention that it is clear that directed edges are not relevant for exchangeable random networks.
If we vary the type of edge so that edges are either all bidirected or
undirected, this leads to six different dependence structures since the
undirected and bidirected graphs in cases (a) and (d) are Markov equivalent. If
an exchangeable random network $X$ satisfies the Markov property associated with
an undirected or bidirected dependence graph $G$ with vertex set $ \{ X_{ij}, ij
\in D(\nodeset) \}$ then $G$ is one of these cases. Dependence structures corresponding to the complete graph in (d) are uninteresting. The case (a) corresponds to the Erd\"os--Renyi model. Structures corresponding to the other four dependence graphs are classified below, where our main focus is on the bidirected incidence graph case as it corresponds to the dissociated model and has certain desired properties concerning extendability; see Section \ref{sec:consistent}.

\paragraph{Undirected incidence graph -- the Frank-Strauss model.}
\cite{fra86} showed the result below:
 \begin{prop}
Every clique $C$ of the incidence graph $L(\nodeset)$ of the complete network base $(\nodeset, D)$ corresponds exactly to a subnetwork $(\nodeset_C,C)$ of $(\nodeset, D)$ that is either a triangle or a $k$-star.
\end{prop}
  Based on this, they show that a random network $X$ is Markov with respect to $L(\nodeset)$ if and only if the canonical parameters $\nu_B$ are zero unless $B$ is a $k$-star or a triangle. Assuming also exchangeability, this model is known as the Frank-Strauss model and studied extensively in \cite{fra86}.

\begin{example}
Consider again the network $x$ 
in Figure~\ref{fig:3} and any choice of
the canonical parameters $\nu \in
\mathbb{R}^{\mathcal{U}_{\nodeset} \setminus \varnothing}$.
Following Example~\ref{ex:exch}, the probability of observing $x$ under the Frank--Strauss model becomes
\begin{equation}\label{eq:xfrank}
    P_{\nu}(x)=\exp\left\{4\nu_{\n{\vcenter{\hbox{\includegraphics[scale=0.05]{figureex2-1.pdf}}}}}+5\nu_{\n{\vcenter{\hbox{\includegraphics[scale=0.05]{figureex2-2.pdf}}}}}+\nu_{\n{\vcenter{\hbox{\includegraphics[scale=0.05]{figureex2-4.pdf}}}}}+\nu_{\n{\vcenter{\hbox{\includegraphics[scale=0.05]{figureex2-5.pdf}}}}}-\psi(\nu)\right\},
\end{equation}
whereas no further simplification appears in terms of the mean-value  parameters
$\{ z_U, U \in \mathcal{U}_{\nodeset} \}$.
\end{example}

\paragraph{Undirected complement of the incidence graph.}

Cliques in the complement of the incidence graph $L^c(\nodeset)$ correspond to collections of disjoint dyads. Therefore, the corresponding model is an ERGM with sufficient statistics corresponding to collections of disjoint dyads (including the subnetwork that is a single dyad). For the case of a network with five nodes, the dependence graph is the famous \emph{Petersen graph} \citep{petersen:98} and for $n$ vertices, the graph is known as the Kneser graph $KG_{n,2}$ \citep[Ch.\ 7]{godsil:royle:01}; for this reason we shall refer to the model determined by all exchangeable random networks that are Markov w.r.t.\ the Kneser graph as the \emph{Kneser model}.

\begin{example}\label{ex:undirectedcomplement}
Consider again the network $x$ in Figure~\ref{fig:3}. Still following Example~\ref{ex:exch} we get
using the exponential representation (\ref{eq:expexch}) and the Markov property w.r.t.\ $L^c(\nodeset)$  that in the Kneser model,
the probability of observing $x$ is
\begin{equation}\label{eq:undircomp}
    P_{\nu}(x)=\exp\left\{4\nu_{\n{\vcenter{\hbox{\includegraphics[scale=0.05]{figureex2-1.pdf}}}}}+\nu_{\n{\vcenter{\hbox{\includegraphics[scale=0.05]{figureex2-3.pdf}}}}}-\psi(\nu)\right\},
\end{equation}
for any $\nu \in
\mathbb{R}^{\mathcal{U}_{\nodeset} \setminus \varnothing}$.
\end{example}

\paragraph{Bidirected complement of the incidence graph.}  Every connected
subset of the complement of the incidence graph $L^c_{\leftrightarrow}(\nodeset)$ corresponds to a disconnected subnetwork of the of size $\nodeset$ or the subnetwork that is a single tie (which corresponds to a single vertex in $L^c_{\leftrightarrow}(\nodeset)$). Therefore, the corresponding model has M\"obius parameters corresponding to these. This model satisfies a property that is dual to dissociatedness so that, for example $X_{ij}\cip X_{jk}$.

\begin{example}\label{ex:bidirectedcomplement} If the random network $X$ is Markov w.r.t.\ the complement of the bidirected incidence graph $L^c_{\leftrightarrow}(\nodeset)$,
the M\"obius parametrization of the probability of the network $x$ in Figure~\ref{fig:3} is
\begin{equation}\label{eq:bidircompl}
    \mathbb{P}\left( X=x \right) =(z_{\n{\vcenter{\hbox{\includegraphics[scale=0.05]{figureex2-1.pdf}}}}})^2z_{\n{\vcenter{\hbox{\includegraphics[scale=0.05]{figureex2-3.pdf}}}}}-2 z_{\n{\vcenter{\hbox{\includegraphics[scale=0.05]{figureex2-1.pdf}}}}}(z_{\n{\vcenter{\hbox{\includegraphics[scale=0.05]{figureex2-3.pdf}}}}})^2+(z_{\n{\vcenter{\hbox{\includegraphics[scale=0.05]{figureex2-3.pdf}}}}})^3.
\end{equation}
This has been obtained, for example, as
$z_{\n{\vcenter{\hbox{\includegraphics[scale=0.05]{figureex2-7.pdf}}}}}=(z_{\n{\vcenter{\hbox{\includegraphics[scale=0.05]{figureex2-1.pdf}}}}})^2z_{\n{\vcenter{\hbox{\includegraphics[scale=0.05]{figureex2-3.pdf}}}}}$ since, for an arbitrary labeling of ${\vcenter{\hbox{\includegraphics[scale=0.05]{figureex2-7.pdf}}}}$, two of the three disconnected induced subgraphs $L^c_{\leftrightarrow}(\nodeset)[12,34]$, $L^c_{\leftrightarrow}(\nodeset)[13,24]$, and $L^c_{\leftrightarrow}(\nodeset)[14,23]$ of $L^c_{\leftrightarrow}(\nodeset)$ correspond to the subnetwork ${\vcenter{\hbox{\includegraphics[scale=0.05]{figureex2-1.pdf}}}}$ and one to ${\vcenter{\hbox{\includegraphics[scale=0.05]{figureex2-3.pdf}}}}$.
\end{example}
Later in Section~\ref{sec:consistent} we will argue that
the only Markov structures that are also extendable to arbitrarily large
networks are the complete graph, the empty graph, and the bidirected incidence
graph; hence the latter of these structures is the main focus of the present paper.

\subsection{Bidirected incidence graph -- the dissociated model}\label{sec:exc-dis}
We now consider in detail the class of finitely exchangeable network distributions that
are Markov with respect to the bidirected incidence graph
$G=L_{\leftrightarrow}(\nodeset)$, where $\nodeset$ is finite. This is  equivalent to requiring an exchangeable network $X$ on $\nodeset$ to be dissociated:
 $\{X_{ij}, i,j\in
\nodeset'\}$ are independent of $\{X_{kl}, k,l\in \nodeset''\}$ for each pair
disjoint subsets $\nodeset'$ and $\nodeset''$ of $\nodeset$.

We begin by stating a simple yet useful property of the incidence graph. For a given finite node set $\nodeset$, let
$C(\nodeset) \subset D(\nodeset)$ denote the subsets of dyads corresponding to
connected subgraphs of the complete graph on $\nodeset$.

\begin{lemma}\label{coro:dissocnet}
A random network $X$ is dissociated if and only if for any non-empty subnetwork $B$ it holds that
\[
    z_B =  \prod_i z_{C_i},
\]
where the $C_i$ are the connected components of $B$.
\end{lemma}
\begin{proof}
The proof follows directly from (\ref{eq:prod_rest}) and the fact that there is a one-to-one correspondence between the set of all connected
    subgraphs of  $L_{\leftrightarrow}(\nodeset)$  and  $C(\nodeset)$.
\halm \end{proof}

With some abuse of notation, for any non-empty $U \in \mathcal{U}_{\nodeset}$,
we will write $U=\dot{\cup} C$ for the decomposition of one  arbitrary graph in the
isomorphism class $[U]$
into the disjoint union of its connected components in $C(\nodeset)$. Note that
each graph in $[U]$ has  the same number of isomorphic connected
components.
By exchangeability and dissociatedness, for any $U \in \mathcal{U}_{\nodeset}$,
\[
z_U = \prod_{C\in C(\nodeset):U=\dot{\cup} C}z_C,
\]
regardless of the choice of the representing graph in $[U]$.
From (\ref{eq:prod_rest}) and (\ref{eq:mobexch}) we thus arrive at the following result.
\begin{prop}\label{prop:5} A network $X$ on $\nodeset$ is
exchangeable and dissociated if and only if
\begin{equation}\label{eq:100nn}
    \mathbb{P}(X = x)=\sum_{U\in
    	\mathcal{U}_\nodeset}(-1)^{|U|-|x|}r_U(x)\prod_{C\in C(\nodeset) :U=\dot{\cup} C}z_C, \quad \forall x
\in \mathcal{G}_{\nodeset},
\end{equation}
\end{prop}
It is worth noting that the model combining dissociatedness and exhangeability
remains a \emph{log-mean linear model} in the sense of \cite{roverato:etal:13}.
\begin{example}\label{ex:dissexch}

The additional restriction of dissociatedness does not yield any helpful
simplification of the exponential representations 
in
(\ref{eq:11})  from \Cref{ex:exch}, since the dissociatedness
conditions entails restriction on the possible values of the canonical parameters
and not on the form of the probabilities.
For the observed network $x$, no further simplification of its M\"{o}bius
parametrization in  \eqref{eq:xmobexch} is possible.

Continuing with Example \ref{ex:exch}, under dissociatedness, the likelihood
function in (\ref{eq:xmobexch}) should now be maximized not only under the
constraints that all expressions in (\ref{eq:mobius}) be non-negative, but also
under the  non-linear constraints described in (\ref{eq:prod_rest}), which in
this case reduces to just one constraint $z_{\n{\vcenter{\hbox{\includegraphics[scale=0.05]{figureex2-3.pdf}}}}}= z_{\n{\vcenter{\hbox{\includegraphics[scale=0.05]{figureex2-1.pdf}}}}}^2$. Thus, although the likelihood function is linear in $z$, numerical maximization is in general necessary. In this particular case we find that
\[\begin{split}
    \hat z_{\n{\vcenter{\hbox{\includegraphics[scale=0.05]{figureex2-1.pdf}}}}}=1/2,\; \hat z_{\n{\vcenter{\hbox{\includegraphics[scale=0.05]{figureex2-2.pdf}}}}}=5/16, \hat z_{\n{\vcenter{\hbox{\includegraphics[scale=0.05]{figureex2-3.pdf}}}}}=1/4\;\hat z_{\n{\vcenter{\hbox{\includegraphics[scale=0.05]{figureex2-4.pdf}}}}}= 3/16,\\ \hat z_{\n{\vcenter{\hbox{\includegraphics[scale=0.05]{figureex2-5.pdf}}}}}= 3/16,\;\hat z_{\n{\vcenter{\hbox{\includegraphics[scale=0.05]{figureex2-6.pdf}}}}}= 1/8,\;\hat z_{\n{\vcenter{\hbox{\includegraphics[scale=0.05]{figureex2-7.pdf}}}}}=1/16,\end{split}
\]
and the remaining parameter estimates are zero,
which should be compared to the results in the non-dissociated case.
In fact, this estimate represents a mixture of the uniform distribution of all
networks isomorphic to the observed network $x$, and the empty network, with
weights $3/4$ and $1/4$, respectively.
\end{example}
We emphasize that when dissociatedness is assumed, the maximum likelihood estimator may no longer be unique, as the example below shows.

\begin{figure}[htb]
\centering
\begin{tikzpicture}[node distance = 5mm and 5mm, minimum width = 3mm]
    \begin{scope}
      \tikzstyle{every node} = [shape = circle,
      font = \scriptsize,
      minimum height = 3mm,
      inner sep = 2pt,
      draw = black,
      fill = white,
      anchor = center],
      text centered]
      \node(a) at (0,0) {$1$};
      \node(b) [right = of a] {$2$};
      \node(c) [right= of b] {$3$};
			\node(d) [right = of c] {$4$};

    \end{scope}
    \begin{scope}
      \draw (a) -- (b);
      \draw (b) -- (c);
			\draw (c) -- (d);
			
    \end{scope}
    \end{tikzpicture}
  \caption[]{\small{An observed networks with four nodes.}}
     \label{fig:line}
\end{figure}
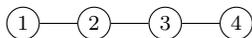
\begin{example}
Suppose that we observe the network  in Figure~\ref{fig:line}.
Then the likelihood function assuming dissociatedness is maximized at any value
of $\lambda$ satisfying $0\leq \lambda\leq 1/16$ by the quantities
\[\hat z_{\n{\vcenter{\hbox{\includegraphics[scale=0.05]{figureex2-1.pdf}}}}}=1/2,\; \hat z_{\n{\vcenter{\hbox{\includegraphics[scale=0.05]{figureex2-2.pdf}}}}}=3/16,\;\hat z_{\n{\vcenter{\hbox{\includegraphics[scale=0.05]{figureex2-3.pdf}}}}}=1/4,\;\hat z_{\n{\vcenter{\hbox{\includegraphics[scale=0.05]{figureex2-4.pdf}}}}}= 1/16-\lambda,\; \hat z_{\n{\vcenter{\hbox{\includegraphics[scale=0.05]{figureex2-5.pdf}}}}}= \lambda,\;\hat z_{\n{\vcenter{\hbox{\includegraphics[scale=0.05]{figureex2-6.pdf}}}}}= 1/16,\]
and all other $\hat{z}$ equal to zero. Indeed this corresponds to a random network that has probability $3/4$ of being isomorphic to the observation and the remaining probability mass of $1/4$ is distributed arbitrarily  between a triangle plus an isolated point, and a 3-star.
\end{example}

\section{Exchangeable and summarized random networks}\label{sec:summarized}

An assumption commonly used in network modeling is that the probability of a
network configuration $x$ in $\mathcal{G}_{\nodeset}$ be a function of its
{\it degree distribution\/} $\{ n_j(x), j \in
\nodeset \}$, where
\[
n_j(x)= \sum_{i\in \nodeset(x)}\delta_j\{d_i(x)\}, \quad j \in \nodeset,
\]
with $d_i(x) = \sum_{j \in \nodeset }X_{ij}$ the degree of node $i$ and
 $\delta_j(k)=1$ if $k=j$ and $0$ otherwise. Note that here the degree distribution is not normalized and $n_j(x)$ is the number of nodes in $x$ with degree $j$.

 \cite{lauritzen:08} defines a random binary array $X=(X_{ij})_{i,j\in\nodeset}$ to be \emph{weakly summarized} (WS)  if there is a function $\varphi$ such that  it holds that
\begin{equation}
    \mathbb{P}\{(X_{ij}=x_{ij})_{i,j\in\nodeset}\}=\varphi(R_i+C_i,i\in \nodeset),
\end{equation}
where $R_i=\sum_jX_{ij}$ and $C_j=\sum_iX_{ij}$ are the row and column sums. For symmetric arrays, it holds that $R_i=C_i$, and a distribution with the above property may be called \emph{symmetric weakly summarized} (SWS).
Again, this can be translated into statements about networks:  A random  network
$X$ on $\nodeset$ is \emph{summarized} if, for any $x \in
\mathcal{G}_{\nodeset}$,  $\mathbb{P}(X=x)=\varphi(d_i(x),i\in\nodeset)$, where $d_i(x)=R_i=C_i$ is the degree of node $i$.
SWE distributions are generally not SWS or vice versa \citep{lauritzen:08}.
If a distribution is both SWE and SWS then we write that the distribution is SWES.
When $\nodeset$ is infinite and $X$ is a SWES array whose distribution
is in $\mathcal{E}_\infty$, \cite{lauritzen:08} showed that \Cref{prop:1n} holds
with $\phi$  having a specific form: $\phi(u,v)={a(u)a(v)}/\{1+a(u)a(v)\}$ \citep{lauritzen:08}.
The  beta model in \eqref{eq:beta} then emerges when further conditioning on $u$ and letting $\beta_i=\log a(u_i)$.

\subsection{Characterizing summarizing statistics}

We now study exchangeable and summarized distributions on a finite node set
$\nodeset$. Isomorphic graphs have the same degree distribution; but, as
shown for instance in Example \ref{ex:7} below, the opposite is not true and thus the set of
summarized and exchangeable network distributions on $\nodeset$ is a subfamily
of the set of exchangeable distribution on $\nodeset$. In order to characterize
the types of restriction that summarizedness entails, we begin with some
preliminary lemmas which identify which subgraph counts are functions of the
degree distribution only. Recall that subgraph counts are canonical sufficient
statistics for the exponential family of exchangeable network
distributions: see \Cref{sec:exch.parametrization}.

\begin{lemma}\label{lem:3}
    Given an unlabeled graph $U\in\mathcal{U}_\nodeset$, it holds that $\sigma_U(x)$ 
    is a function of the degree distribution of $x$,  for all $x \in
    \mathcal{G}_{\nodeset}$, if
\begin{enumerate}[\rm(a)]
 \item $U=S_k$, a $k$-star for any size $k\geq 0$ or
     $U=\n{\vcenter{\hbox{\includegraphics[scale=0.05]{figureex2-3.pdf}}}}$, the
     disjoint union of two independent edges, or
 \item $U=K_{|\nodeset|-1}$, the complete graph of size $|\nodeset|-1$, or
     $U=K_{|\nodeset|-1}\setminus\{e\}$, the complete graph of size $|\nodeset|-1$ with one edge removed.
\end{enumerate}
\end{lemma}
\begin{proof}
To prove (a), note that a $0$-star is simply a vertex and a $1$-star is  an edge.
If $U=S_k$ for $k\neq 1$ then
\begin{equation*}
\sigma_{S_k}(x)=\sum_{j=0}^\infty{j\choose k}n_j(x),
\end{equation*}
where the sum only effectively extends at most from $j=k$ to $| \nodeset|-1$ as all other terms are zero.
If $U=S_1$ we have to divide by two, and obtain
\begin{equation}\label{eq:handshake}
\sigma_{S_1}(x) =\sum_{j=0}^\infty jn_j(x)/2 =|E(x)|
\end{equation}
where $|E(x)|$ is the number of edges in $x$. The formula (\ref{eq:handshake}) is also known as Euler's \emph{handshaking lemma}; see e.g.\ \cite[p.12]{wilson:96}. Finally,
if $U=\n{\vcenter{\hbox{\includegraphics[scale=0.05]{figureex2-3.pdf}}}}$   we have
\begin{equation*}
\sigma_{\n{\vcenter{\hbox{\includegraphics[scale=0.05]{figureex2-3.pdf}}}}}  ={|E(x)|\choose 2}-\sigma_{S_2}(x).
\end{equation*}
For (b), we have that
$$\sigma_{K_{|\nodeset|-1}}(x)=
\left\{                                 																
\begin{array}{ll}
  |\nodeset|, &\text{if } \mbox{$n(x)=(0,\dots,0,|\nodeset|)$;} \\
  2, & \text{if }\mbox{$n(x)=(0,\dots,0,2,|\nodeset|-2)$;} \\
                                      1, &\text{if } \mbox{$n(x)=(0,\dots,0,1_r,0,\dots,0,|\nodeset|-r-1,r)$;} \\
                                      0, & \mbox{\text{otherwise},}
                                    \end{array}
                                  \right. $$
where $1_r$ is $1$ at the $r$th entry. The top row corresponds to the complete graph, the second row to the complete graph with one edge removed, and the third row to a complete graph with a vertex from which $|\nodeset|-r-1$ edge is removed. 

A similar function can be provided for the case of $K_{|\nodeset|-1}\setminus\{e\}$, but we abstain from giving the somewhat cumbersome details. 

\halm \end{proof}

The converse to the statement in Lemma~\ref{lem:3} only holds with some modification: 
\begin{lemma}\label{lem:onlyif}
    Given an unlabeled graph $U\in\mathcal{U}_\nodeset$, if $\sigma_U(x)$ 
    is a function of the degree distribution of $x$,  for all $x \in
    \mathcal{G}_{\nodeset}$, then
\begin{enumerate}[\rm(a)]
 \item if $U$ has at most $|\nodeset|-2$ vertices, $U$ is a $k$-star or
     $U=\n{\vcenter{\hbox{\includegraphics[scale=0.05]{figureex2-3.pdf}}}}$, the disjoint union of two independent edges;
 \item if $U$ has at most $|\nodeset|-1$ vertices, $U$ is a $k$-star, or $U=
     \n{\vcenter{\hbox{\includegraphics[scale=0.05]{figureex2-3.pdf}}}}$, or
     $K_{|\nodeset|-1}$, or $K_{|\nodeset|-1} \setminus e$, a complete graph on
     $|\nodeset|-1$ vertices with one edge removed.
\end{enumerate}

\end{lemma}
\begin{proof}
It is enough to show that for any $U$ that is not one of the above graphs, there exist two graphs $x_1$ and $x_2$ with the same degree sequence (and hence the same degree distribution) such that $\sigma_U(x_1)\neq \sigma_U(x_2)$.

(a): For $U$ with at most $|\nodeset|-2$ vertices, if $U$ contains two adjacent vertices $i$ and $j$ with degrees more then $1$ then define $x_1$ and $x_2$ as follows: Let $x_1$ contain $U$ and a disjoint edge $kl$. Let $x_2$ contain $U$ with the $ij$ edge removed and two non-adjacent vertices $k$ and $l$ adjacent to $i$ and $j$ respectively. These two graphs obviously have the same degree sequence and  $\sigma_U(x_1)=1$ whereas $\sigma_U(x_2)=2$.

Hence, the cases remaining to be considered are such where $U$ is a forest of stars. If there are at least two stars in $U$ that are $2$-stars or larger then call these $S$ and $T$. Let $x_1$ be $U$ and define $x_2$ as follows. Connect the hubs of $S$ and $T$. In addition, take a leaf from each of $S$ and $T$, connect them and remove the edge between them and their respective hubs. Now $x_1$ and $x_2$ have the same degree sequence but $\sigma_U(x_1)=1$ and $\sigma_U(x_2)=0$.

If $U$ contains only one star $S$ with hub $s$ that is not an edge and some disjoint edges then take an edge $ij$ in $U$ which is disjoint from $S$. Define $x_1$ by adding a vertex $l$ to $U$ and connecting it to $i$ and to a leaf $k$ of $S$. Define $x_2$ by taking $x_1$, removing the $sk$ edge and the $ij$ edge, and adding the $ki$ edge and the $sj$ edge.
Both graphs have the same degree sequence, but if $S$ is a $3$-star or larger we have $\sigma_U(x_1)=2$ and $\sigma_U(x_2)=3$; if $S$ is a $2$-star we have $\sigma_U(x_1)=6$ and $\sigma_U(x_2)=9$.

Finally, the only case that is left is when $U$ is collection of three or more disjoint edges:  $ij$, $kl$, $hr$, and a set  $\mathcal{E}$ of disjoint edges. Let $x_1$ be the union of the cycle $i\sim k\sim l\sim j\sim i$, the $hr$ edge, and $\mathcal{E}$;  let $x_2$ be the union of the path $h\sim i\sim j\sim l\sim k\sim r$ and $\mathcal{E}$.  Both graphs have the same degree sequence but $\sigma_U(x_1)=2$ whereas $\sigma_U(x_2)=1$.

(b): By (a) it is enough to work with $U$ with $n-1$ vertices that is not one of those mentioned. Collections of disjoint edges and stars (including forest of edges) have been covered in the proof of (a). If $U$ is disconnected and not one of these, then take two connected components with edges $ij$ and $kl$ in each component. Let $x_1$ be $U$, and define $x_2$ by removing $ij$ and $kl$ and connecting $i$ to $k$ and $j$ to $l$. We have that $\sigma_U(x_1)=1$ whereas $\sigma_U(x_2)=0$.

Hence, henceforth, assume that $U$ is connected. If there is an induced $P_4$ or $C_4$ $(i,j,k,l)$ in $U$, i.e., an induced path or cycle with $4$ nodes, then define $x_1$ and $x_2$ as follows: For $x_1$, add a vertex $h$ and connect it to $i$. For $x_2$, add $h$ and connect it to $l$, remove $kl$, and add $ik$. The graphs have the same degree sequence, but $\sigma_U(x_1)>0$ whereas $\sigma_U(x_2)=0$.

If there are no induced $P_4$ or $C_4$ the graph is trivially perfect (see \citet{brandstadt:etal:99} Section 7.1.2)  and there is a vertex $i$ that is adjacent to all other vertices in $U$. 

If among the neighbours of $i$, there are two non-adjacent pairs of vertices and an edge we have the following two cases:

1) Among the neighbours of $i$, suppose that there are two non-adjacent pairs of vertices $(j,k)$ and $(k,l)$ and a pair of adjacent vertices $(l,h)$, where $h$ may be the same as $j$. In this case, define $x_1$ by adding $u$ and connecting it to $k$, and define $x_2$ by adding $u$ and connecting it to $h$, removing $lh$, and connecting $kl$. The graphs have the same degree sequence, but $\sigma_U(x_1)>0$ whereas $\sigma_U(x_2)=0$.

2): Among the neighbours of $i$, suppose that there are two non-adjacent pairs of vertices $(j,k)$ and $(l,h)$ and a pair of adjacent vertices $(k,l)$, where $h$ may be the same as $j$. In this case, define $x_1$ by adding $u$ and connecting it to $h$, and define $x_2$ by adding $u$ and connecting it to $k$, adding $lh$, and removing $kl$. The graphs have the same degree sequence, but $\sigma_U(x_1)>0$ whereas $\sigma_U(x_2)=0$.

Therefore, the only cases that are left are stars, the complete graph, or the complete graph with one edge removed.
This completes the proof.
\halm \end{proof}

In general, if $x_1$ and $x_2$ have the same degree distributions, the fact that
their probabilities should be identical induces a linear restriction on the
parameters $(\nu_U, U\in \mathcal{U}_\nodeset)$ and hence the set of
exchangeable and summarized distributions on a nodeset $\nodeset$ is again a
linear exponential family with dimension equal to the number of distinct degree
distributions for simple graphs on $| \nodeset|$ nodes; however, the restrictions induced on the parameters in (\ref{eq:expexch}) may be complicated, as illustrated in the example below.

Similarly, if $x_1$ and $x_2$ have the same degree distributions, a linear restriction is induced on the parameters $(z_U, U\in \mathcal{U}_\nodeset)$ in (\ref{eq:mobexch}). This is also illustrated in the example below.
\begin{example}\label{ex:7} For networks with at most four nodes, the isomorphism class of a network is uniquely determined by the degree distribution, so we consider networks with five nodes. For most such networks, the isomorphism class is also identified by the degree distribution, the six exceptions being displayed in Fig.~\ref{fig:summarized}, where the node degrees of $x_1$ and $x_2$ both are $(2,2,2,1,1)$, those of $y_1$ and $y_2$ $(3,2,2,2,1)$, and those of $z_1$ and $z_2$ $(3,3,2,2,2)$.

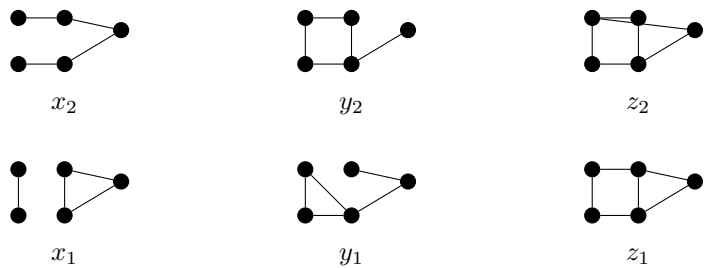
\begin{figure}[htb]
\centering
\begin{tikzpicture}[node distance = 4mm and 4mm, minimum width = 2mm]
    \begin{scope}
      \tikzstyle{every node} = [shape = circle,
      font = \tiny,
      minimum height = 2mm,
      inner sep = 2pt,
      draw = black,
      fill = black,
      anchor = center],
      text centered]
      \node(i) at (0,0) {};
      \node(j) [right = of i] {};
      \node(l) [above = of j] {};
			\node(k) [above = of i] {};
			\node(m) [above right = 3mm and 6mm of j] {};
      \node(i1) [above = 18mm of i] {};
      \node(j1) [right = of i1] {};
      \node(l1) [above = of j1] {};
			\node(k1) [above = of i1] {};
			\node(m1) [above right = 3mm and 6mm of j1] {};
			
			\node(i2) [right = 36mm of i] {};
      \node(j2) [right = of i2] {};
      \node(l2) [above = of j2] {};
			\node(k2) [above = of i2] {};
			\node(m2) [above right = 3mm and 6mm of j2] {};
			
			\node(i3) [above = 18mm of i2] {};
      \node(j3) [right = of i3] {};
      \node(l3) [above = of j3] {};
			\node(k3) [above = of i3] {};
			\node(m3) [above right = 3mm and 6 mm of j3] {};
			
			\node(i4) [left = 36mm of i] {};
      \node(j4) [right = of i4] {};
      \node(l4) [above = of j4] {};
			\node(k4) [above = of i4] {};
			\node(m4) [above right = 3mm and 6mm of j4] {};
			
			\node(i5) [above = 18mm of i4] {};
      \node(j5) [right = of i5] {};
      \node(l5) [above = of j5] {};
			\node(k5) [above = of i5] {};
			\node(m5) [above right = 3mm and 6 mm of j5] {};
			
    \end{scope}
		
    \begin{scope}
    \tikzstyle{every node} = [node distance = 4mm and 4mm, minimum width = 2mm,
    font= \small,
      anchor = center,
      text centered]
\node(a) [below = 2mm  of j]{$y_1$};
\node(b) [below = 2mm  of j1]{$y_2$};
\node(c) [below = 2mm  of j2]{$z_1$};
\node(d) [below = 2mm  of j3]{$z_2$};
\node(e) [below = 2mm  of j4]{$x_1$};
\node(f) [below = 2mm  of j5]{$x_2$};

\end{scope}
    \begin{scope}
      \draw (i) -- (j);
      \draw (j) -- (m);
			\draw (i) -- (k);
			\draw (j) -- (k);
			\draw (l) -- (m);
			
			\draw (i1) -- (j1);
      \draw (j1) -- (m1);
			\draw (i1) -- (k1);
			\draw (j1) -- (l1);
			\draw (l1) -- (k1);
			
			\draw (i2) -- (j2);
      \draw (j2) -- (m2);
			\draw (i2) -- (k2);
			\draw (j2) -- (l2);
			\draw (l2) -- (k2);
			\draw (l2) -- (m2);
			
			\draw (i3) -- (j3);
      \draw (j3) -- (m3);
			\draw (i3) -- (k3);
			\draw (j3) -- (l3);
			\draw (l3) -- (k3);
			\draw (k3) -- (m3);
			
			\draw (i4) -- (k4);
      \draw (j4) -- (l4);
			\draw (j4) -- (m4);
			\draw (l4) -- (m4);
			
			\draw (i5) -- (j5);
      \draw (j5) -- (m5);
			\draw (l5) -- (k5);
			\draw (l5) -- (m5);
    \end{scope}

    \end{tikzpicture}
		\caption{{\footnotesize Three pairs of non-isomorphic graphs with identical degree distributions.}}\label{fig:summarized}
		\end{figure}
		\noindent It thus follows that for any choice of canonical parameters
		$\nu \in \mathbb{R}^{ \mathcal{U}_{\nodeset} \setminus
	    \varnothing }$, the distribution $P_{\nu}$ of an exchangeable network
	    on five nodes is summarized if and only if
	    $P_{\nu}(x_1)=P_{\nu}(x_2)$, $P_{\nu}(y_1)=P_{\nu}(y_2)$, and
	    $P_{\nu}(z_1)=P_{\nu}(z_2)$.  Using the exponential representation
		(\ref{eq:expexch}) we obtain that the canonical parameters must
		satisfy three linear relations. From the relation
		$P_{\nu}(x_1)=P_{\nu}(x_2)$ we get
		\[\log\{P_{\nu}(x_2)/P_{\nu}(x_1)\}=\nu_{\n{\vcenter{\hbox{\includegraphics[scale=0.05]{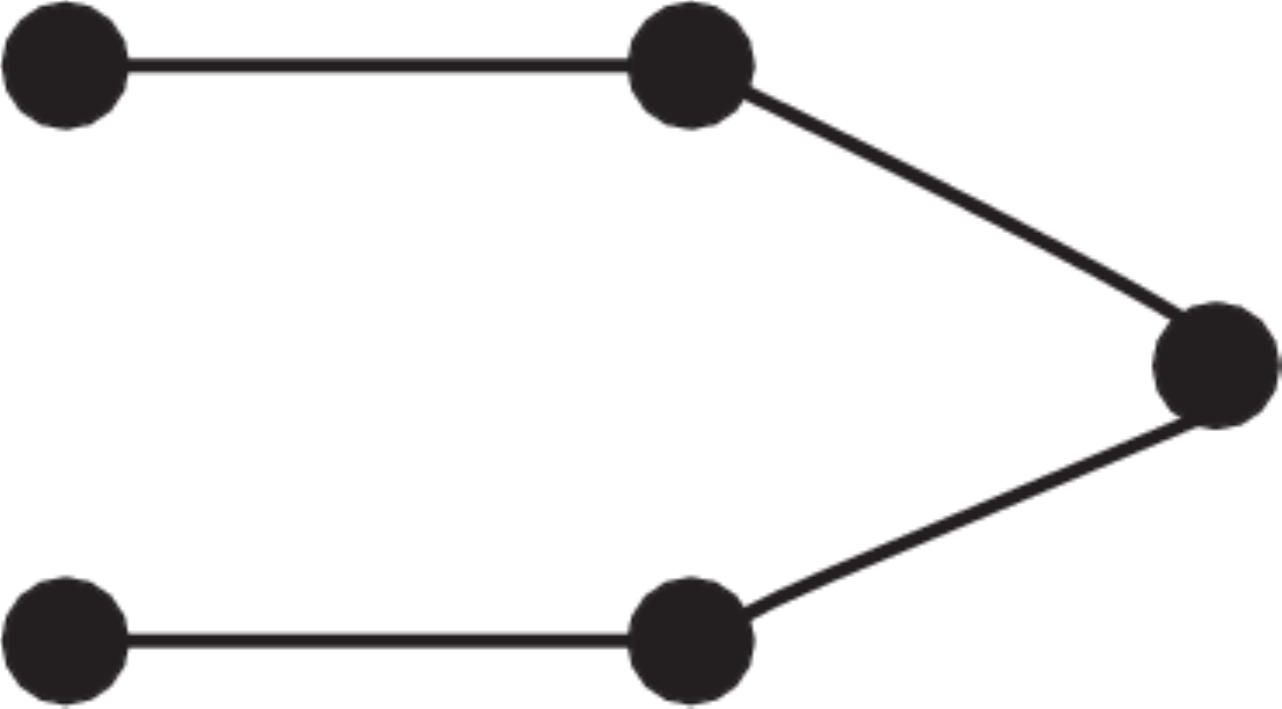}}}}}-\nu_{\n{\vcenter{\hbox{\includegraphics[scale=0.05]{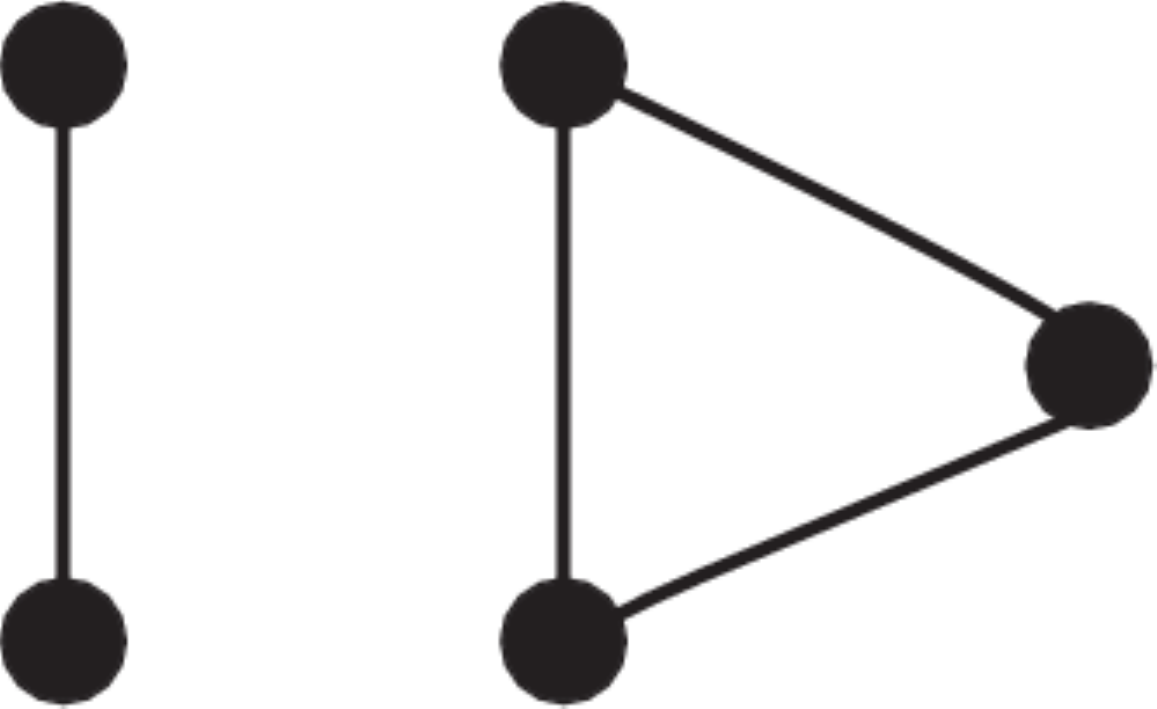}}}}}+2\nu_{\n{\vcenter{\hbox{\includegraphics[scale=0.05]{figureex2-6.pdf}}}}} -\nu_{\n{\vcenter{\hbox{\includegraphics[scale=0.05]{figureex2-4.pdf}}}}}-\nu_{\n{\vcenter{\hbox{\includegraphics[scale=0.05]{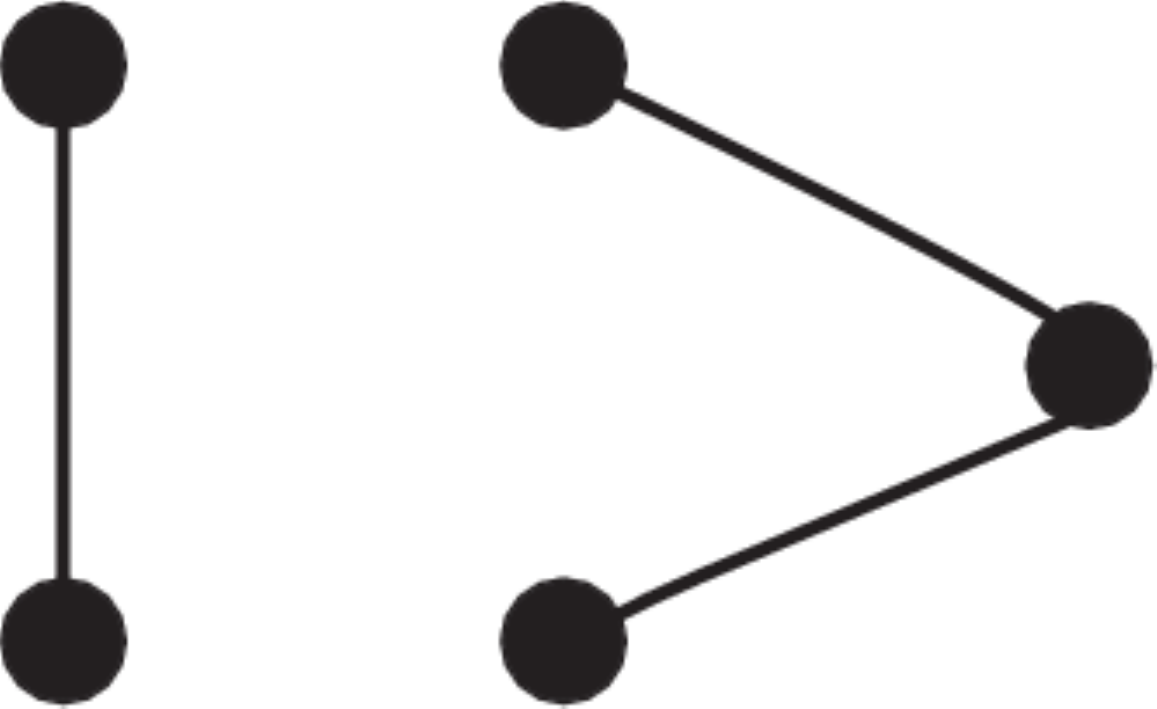}}}}}=0,\]
		where we have used Lemma~\ref{lem:3} to infer that all terms of the form $\nu_U\sigma_U(\cdot)$ must cancel if $U$ is a $k$-star or $U=\n{\vcenter{\hbox{\includegraphics[scale=0.05]{figureex2-3.pdf}}}}$.
		Similarly, we get from
		$P_{\nu}(y_1)=P_{\nu}(y_2)$ that
		\[\log \{P_{\nu}(y_2)/P_{\nu}(y_1)\} =
			\nu_{\n{\vcenter{\hbox{\includegraphics[scale=0.05]{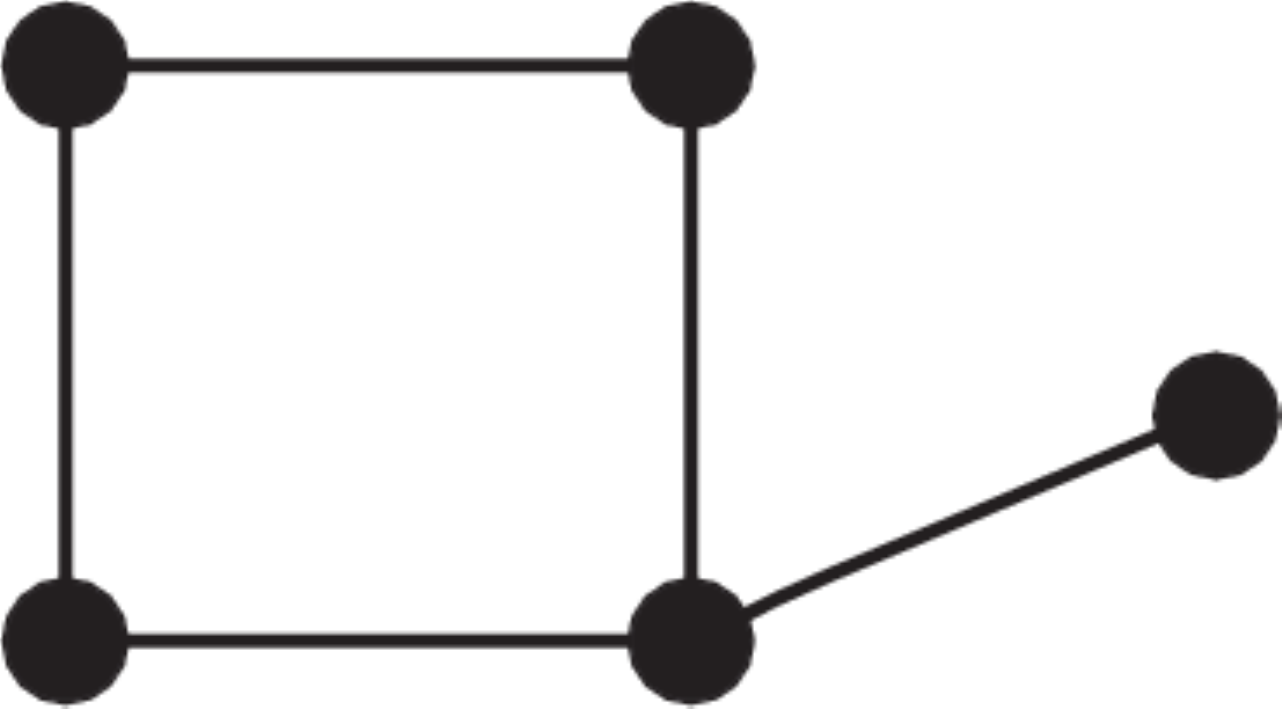}}}}}-\nu_{\n{\vcenter{\hbox{\includegraphics[scale=0.05]{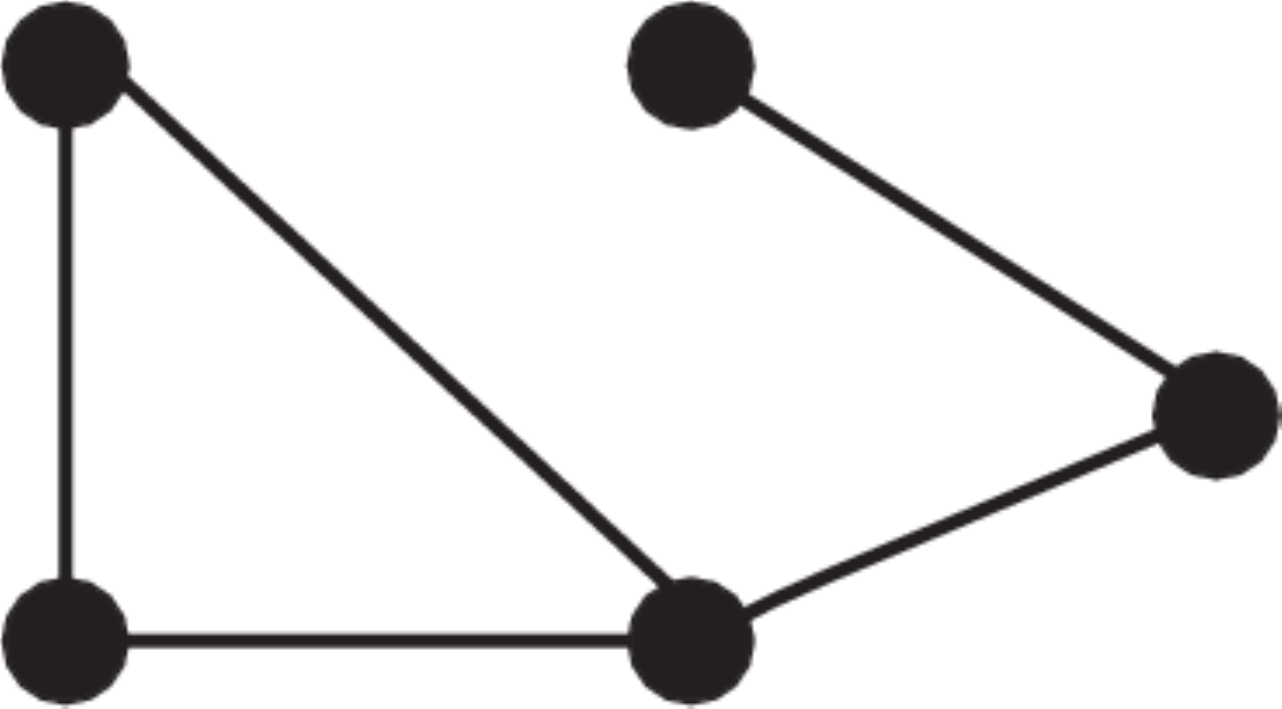}}}}}+\nu_{\n{\vcenter{\hbox{\includegraphics[scale=0.05]{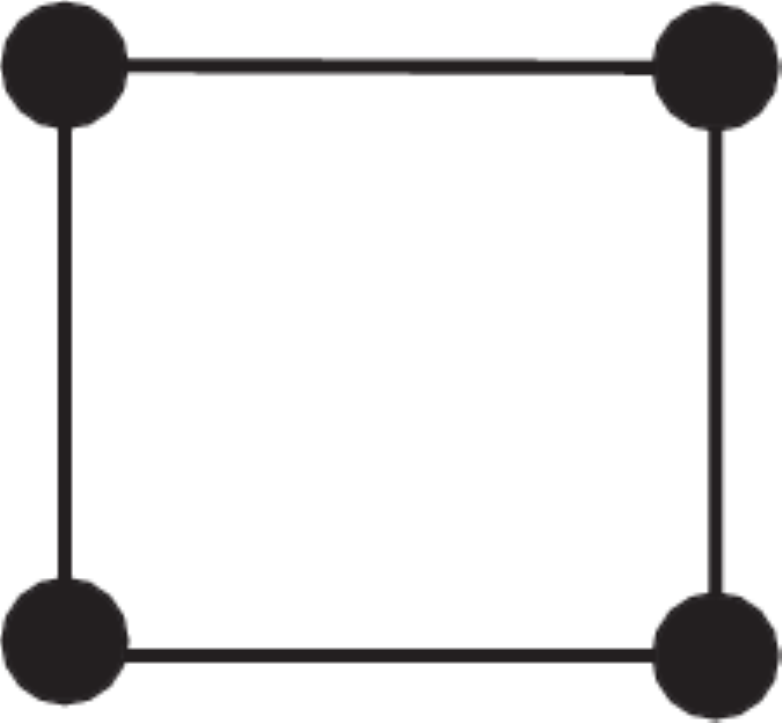}}}}}+\nu_{\n{\vcenter{\hbox{\includegraphics[scale=0.05]{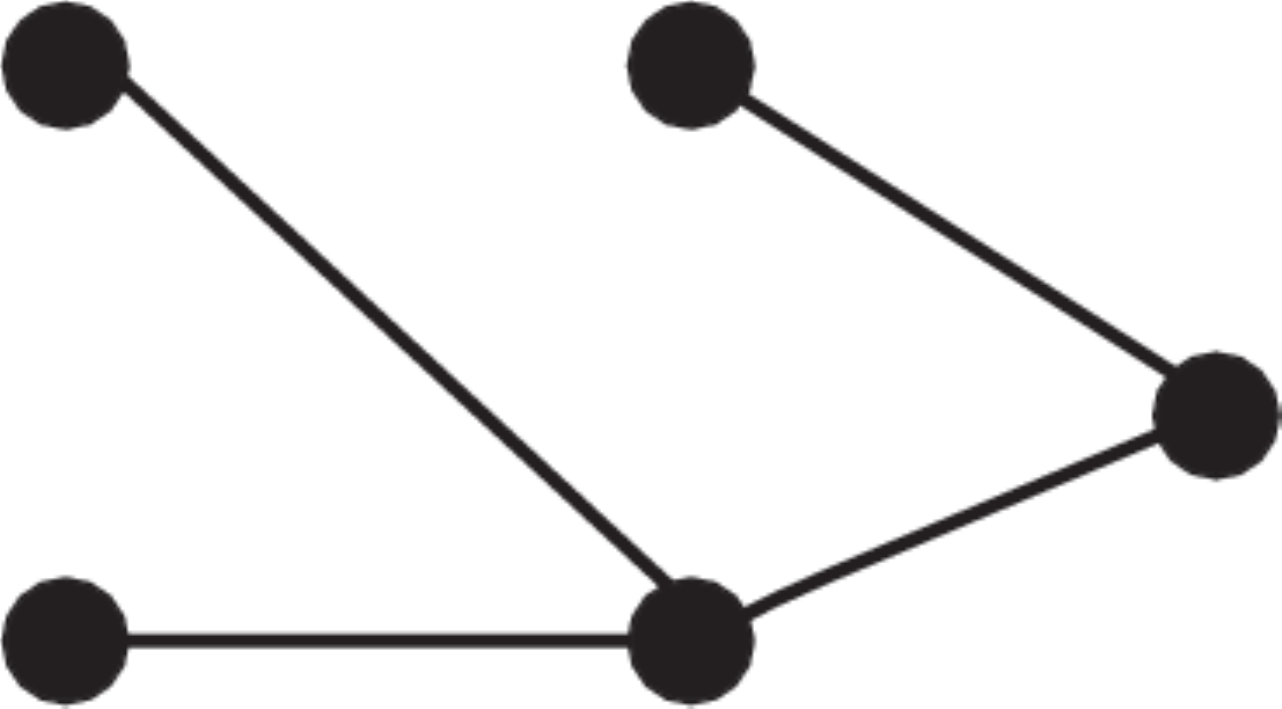}}}}} -\nu_{\n{\vcenter{\hbox{\includegraphics[scale=0.05]{figureex2-7.pdf}}}}} -\nu_{\n{\vcenter{\hbox{\includegraphics[scale=0.05]{figureex7-6.pdf}}}}}-\nu_{\n{\vcenter{\hbox{\includegraphics[scale=0.05]{figureex2-4.pdf}}}}} =0,\]
		and a somewhat more involved linear relation appears from the
		fact that $P_{\nu}(z_1)=P_{\nu}(z_2)$.
Similarly using (\ref{eq:mobexch}) we obtain that the M\"{o}bius parameters must
satisfy three linear relations.  The simplest of the three relations with
M\"{o}bius parameters for $P_{\nu}(z_1)=P_{\nu}(z_2)$ is
\[P_{\nu}(z_2)-P_{\nu}(z_1) =
			z_{\n{\vcenter{\hbox{\includegraphics[scale=0.05]{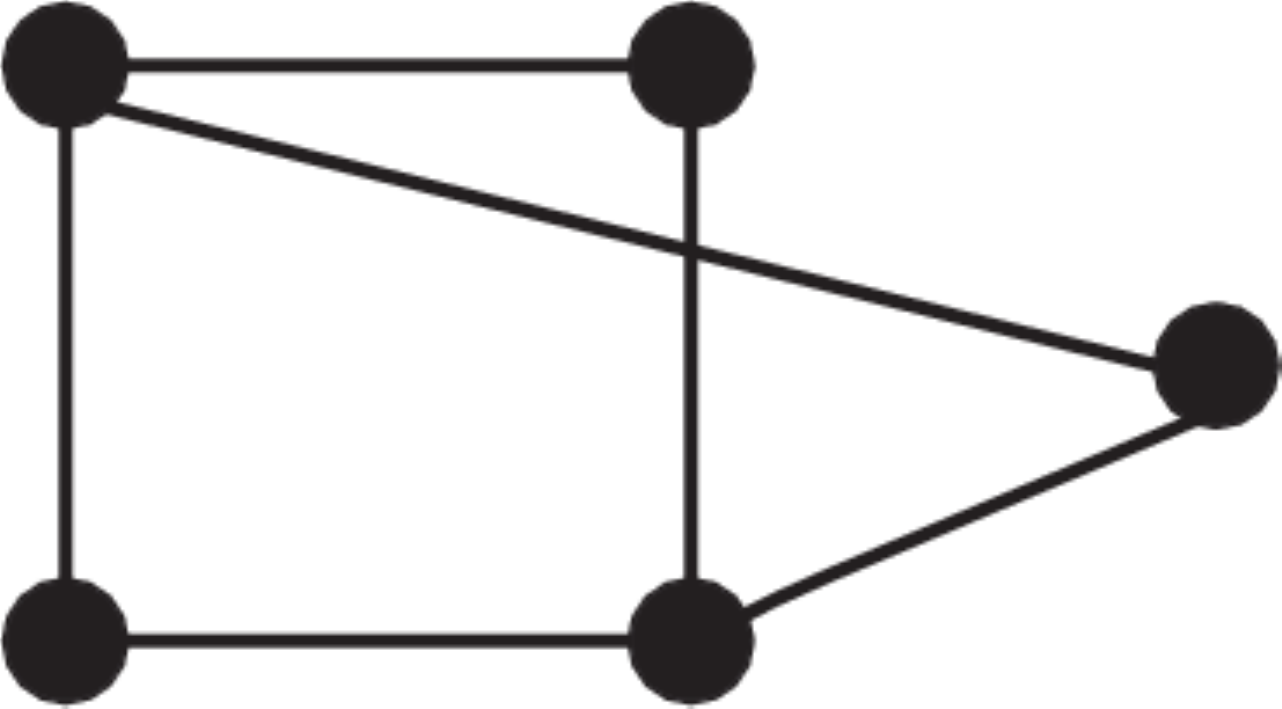}}}}}-z_{\n{\vcenter{\hbox{\includegraphics[scale=0.05]{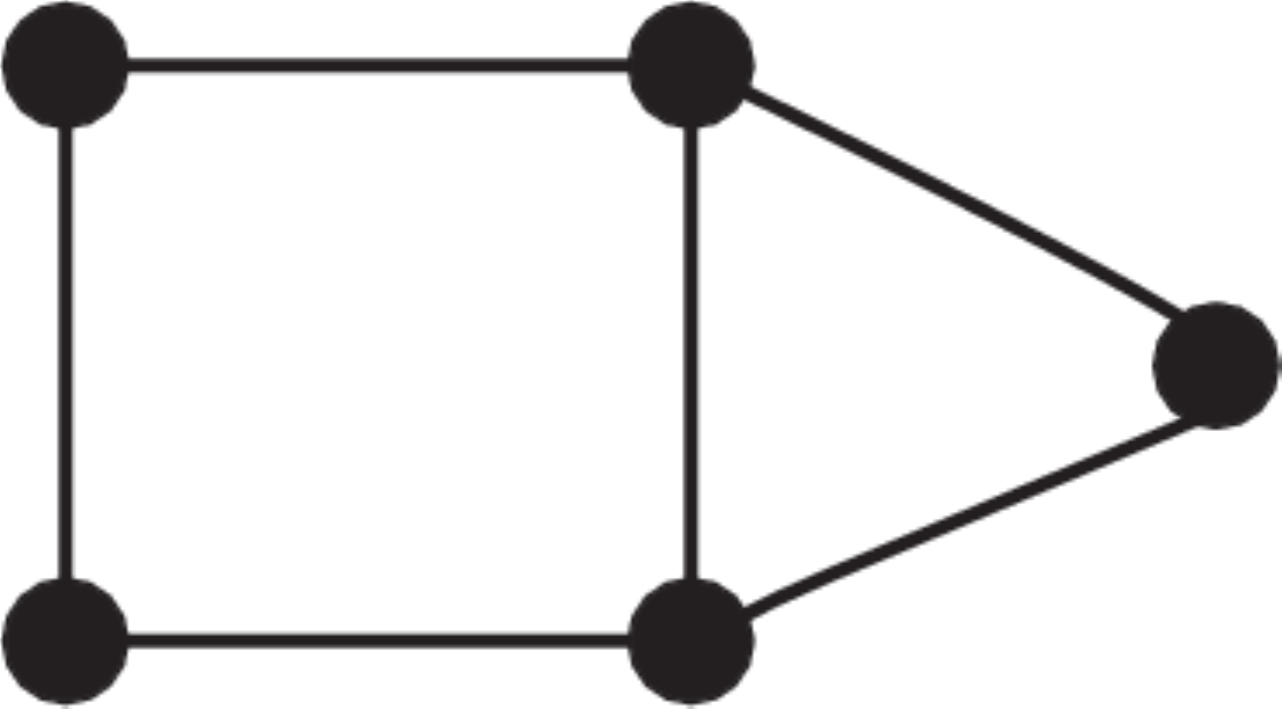}}}}}-z_{\n{\vcenter{\hbox{\includegraphics[scale=0.05]{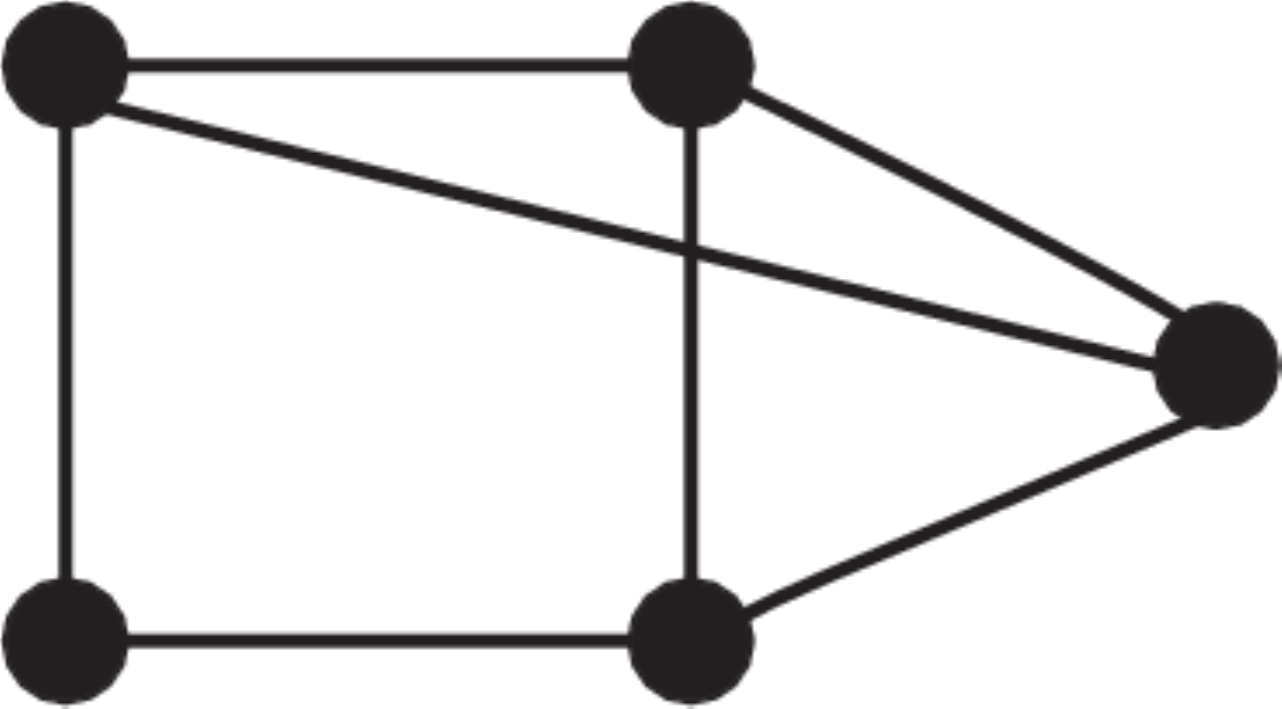}}}}}+z_{\n{\vcenter{\hbox{\includegraphics[scale=0.05]{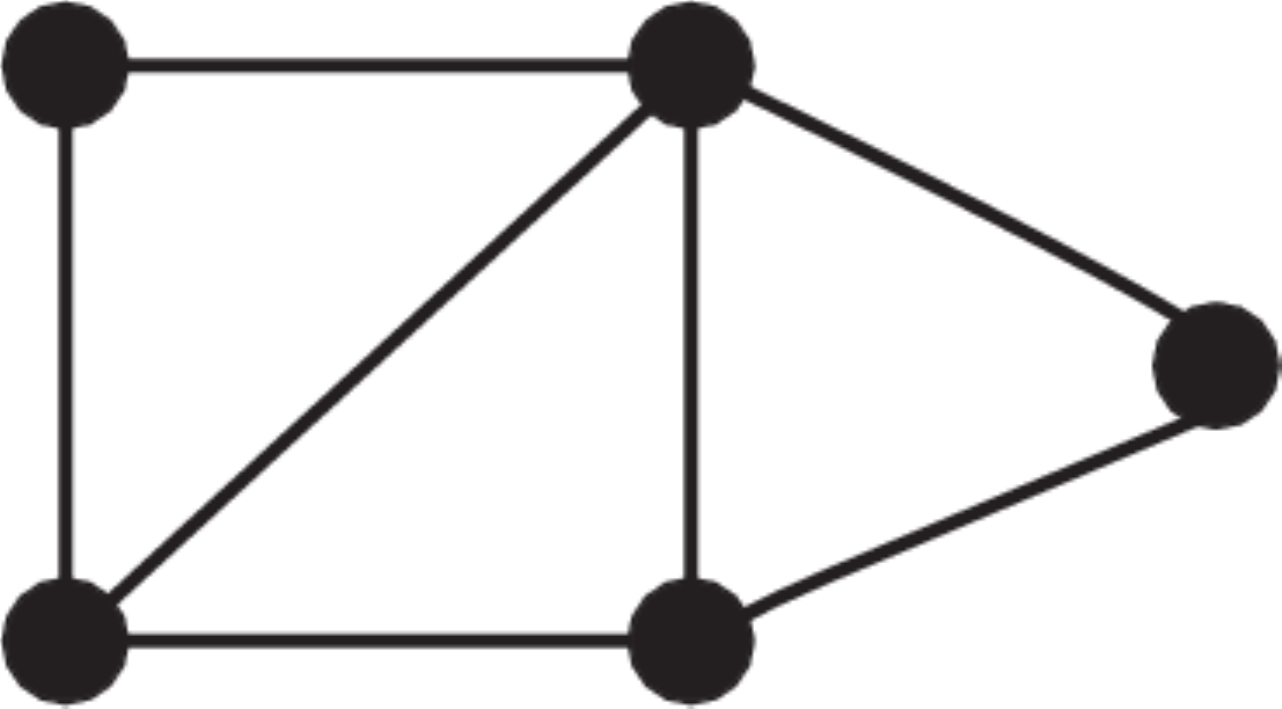}}}}} =0.\]
			We refrain from giving details of the corresponding
			constraints with M\"{o}bius parameters obtained from
			$P_{\nu}(x_1)=P_{\nu}(x_2)$ and
			$P_{\nu}(y_1)=P_{\nu}(y_2)$ as well as with canonical
			parameters obtained from $P_{\nu}(z_1)=P_{\nu}(z_2)$.
		\end{example}

\begin{coro}
Let $U$ be a finite unlabelled graph without isolated nodes. If there is a function
$\varphi_U$ so that $\sigma_U(x)=\varphi_U\{n_j(x),j=0,1,2,\ldots\}$ for all
networks $x$ of any size, then $U=S_k$, a $k$-star, or
$U={\n{\vcenter{\hbox{\includegraphics[scale=0.05]{figureex2-3.pdf}}}}}$, the disjoint union of two independent edges.
\end{coro}
\begin{proof}
If this is true for all networks, we can always choose $\nodeset$ sufficiently large for Lemma~\ref{lem:onlyif}, (a) to apply.
\halm \end{proof}
\subsection{The SE* model}
Motivated by the results from the previous section,
we can now simplify the exchangeable ERGM (\ref{eq:expexch}) and
define an ERGM by the expression
\begin{equation}\label{eq:15}
    P_{\nu}(x)=
    \exp\Big\{\sum_{k=1}^{|\nodeset|-1}\sigma_k(x)\nu_k+\sigma_{\n{\vcenter{\hbox{\includegraphics[scale=0.05]{figureex2-3.pdf}}}}}\,\nu_{\n{\vcenter{\hbox{\includegraphics[scale=0.05]{figureex2-3.pdf}}}}}-\psi(\nu)
\Big\},
    \quad x \in \mathcal{G}_{\nodeset},
\end{equation}
where $\sigma_k(x)$ is the number of $k$-stars in $x$, $1 \leq k \leq
|\nodeset|-1$, and each of the
$|\nodeset|$ parameters
$\nu_{\n{\vcenter{\hbox{\includegraphics[scale=0.05]{figureex2-3.pdf}}}}},
\nu_1,\ldots, \nu_{|\nodeset|-1}$ vary independently on the real line.

We shall refer to this ERGM as the \emph{SE*-model}, reflecting that all
distributions within the model are both summarized and exchangeable, the star
indicating that the model is not containing all such distributions. This is
because neither the number of complete sub-graphs on $|\nodeset| - 1$ nor the
number of complete sub-graphs on $|\nodeset|-1$ nodes with one edge removed are
sufficient statistics for the SE* model, so that, in virtue of \Cref{lem:3} and
\Cref{lem:onlyif}, there exist summarized exchangeable network distributions
that are not in the SE* model. The distributions implied by the marginal beta
model, discussed below, are such an example.
We reserve the term \emph{SE-model} for the set of all distributions that are exchangeable and summarized and \emph{DE-model} for the set of all distributions that are exchangeable and dissociated.

We note the similarities between the SE*-model \eqref{eq:15} and the exchangeable
Markov model (EM-model) of \cite{fra86}, 
exponential family of the form
\begin{equation*}
    P_{\nu}(x)=\exp\left\{\sum_{k=1}^{|\nodeset|-1}\sigma_k(x)\nu_k+
    \sigma_{\n{\vcenter{\hbox{\includegraphics[scale=0.05]{figureex2-4.pdf}}}}}\nu_{\n{\vcenter{\hbox{\includegraphics[scale=0.05]{figureex2-4.pdf}}}}}-\psi(\nu)\right\},
    \quad x \in \mathcal{G}_{\nodeset}, \nu \in \mathbb{R}^{|\nodeset|},
\end{equation*}
which differs from SE*-model only by
$\sigma_{\n{\vcenter{\hbox{\includegraphics[scale=0.05]{figureex2-4.pdf}}}}}\nu_{\n{\vcenter{\hbox{\includegraphics[scale=0.05]{figureex2-4.pdf}}}}}$
replacing
$\sigma_{\n{\vcenter{\hbox{\includegraphics[scale=0.05]{figureex2-3.pdf}}}}}\,\nu_{\n{\vcenter{\hbox{\includegraphics[scale=0.05]{figureex2-3.pdf}}}}}$.
Though nearly identical, the justifications behind  the EM and the SE* models are quite
different. The EM model results from postulating certain symmetric Markov
properties for the edges, covered by case (b) in \Cref{prop:skeleton}.  On the
other hand,
the SE* model arises as a convenient sub-model of the  class of summarized and
exchangeable network distributions.
The intersection of the SE*-model and the EM-model only contains the counts
for $k$-stars in the exponent, and so does the intersection of the SE-model and
the EM
model.  The resulting \emph{SEM-model} is the ERGM with
sufficient statistics given by the entries  degree distribution, studied in
\cite{sadeghi:rinaldo:14}. This is because there is a one-to-one linear transformation between the number of $k$-stars of $x$ and the degree distribution of $x$, as demonstrated in the proof of Lemma~\ref{lem:3}.

From the computational standpoint, parameter estimation in the SE* model and in the
EM model can be carried out by the same procedures, such
as pseudo-maximum likelihood estimation and MCMCMLE methods --  see, e.g.,
\cite{CIS-93373}, \cite{geyer.thompson:1992} and \cite{ergm.package} -- and is
likely to be just as problematic.

\subsection{The marginal beta model}
Here we consider {\it the marginal beta model,} an example of a family of network distributions that are dissociated, exchangeable,
and summarized, and consequently satisfy all the conditions discussed above.

Consider the beta model defined in (\ref{eq:beta}). Suppose  that
$(B_i, i \in \nodeset )$ are independent and identically distributed real-valued random variables with distribution $F_B$. 
Then the expression (\ref{eq:beta}) can be understood as the conditional
distribution of $X_{ij}$ given $B_i$ and $B_j$ and thus
the density satisfies the factorization  $f(y)=\prod_{i}f(y_i\cd \pa(y_i))$ for the
directed acyclic graph (DAG) $D_n$ that consists of $(B_i, i \in \nodeset)$ and
$(X_{ij}, ij \in D(\nodeset))$ as vertices and arrows from each $B_i$ to all the
$X_{ij}$, $j \in \nodeset$.

Figure \ref{fig:1} illustrates the case for
$\nodeset = \{1,2,3,4\}$.
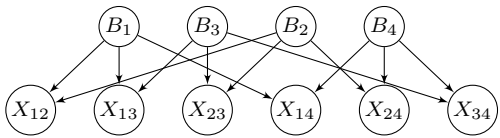
\begin{figure}[htb]
\centering
\begin{tikzpicture}[node distance = 5mm and 5mm, minimum width = 4mm]
    \begin{scope}
      \tikzstyle{every node} = [shape = circle,
      font = \scriptsize,
      minimum height = 4mm,
      inner sep = 1pt,
      draw = black,
      fill = white,
      anchor = center],
      text centered]
      \node(12) at (0,0) {$X_{12}$};
      \node(13) [right = of 12] {$X_{13}$};
			
      \node(23) [right = of 13] {$X_{23}$};
			\node(14) [right = of 23] {$X_{14}$};
			\node(24) [right = of 14] {$X_{24}$};
      \node(34) [right = of 24] {$X_{34}$};
			\node(b1) [above = of 13] {$B_1$};
			\node(b2) [above = of 14] {$B_2$};
			\node(b3) [above = of 23] {$B_3$};
			\node(b4) [above = of 24] {$B_4$};
    \end{scope}

    \begin{scope}[->, > = latex']
      \draw (b1) -- (12);
      \draw (b1) -- (13);
			\draw (b1) -- (14);
			\draw (b2) -- (12);
      \draw (b2) -- (23);
			\draw (b2) -- (24);
			\draw (b3) -- (13);
			\draw (b3) -- (23);
			\draw (b3) -- (34);
			\draw (b4) -- (14);\draw (b4) -- (34);\draw (b4) -- (24);

    \end{scope}

    \end{tikzpicture}
  \caption[]{\small{Directed acyclic graph $D_4$.}}
     \label{fig:1}
\end{figure}
Indeed the exact same graph also illustrates the Markov structure for a general $\phi$-matrix, where just $B_i$ should be replaced by uniformly distributed $U_i$  and the arrows represent the $\phi$s.

By marginalization over and conditioning on the variables in the DAG model the resulting dependence structure can in general be captured by \emph{mixed graphs}; for further details see, for example, \citet{ric02} and \citet{sad13}.
For example, by conditioning on the $B_i$ we obtain the empty graph with all
the $X_{ij}$ as vertices. This implies the complete independence of $X_{ij}$, corresponding to the dyadic independence in the beta model (\ref{eq:beta}).
By marginalizing over $B_i$, the bidirected incidence graph appears, as studied
extensively in this paper and exemplified in Figure~\ref{fig:2}. As discussed
before, this implies that after marginalization over $B_i$, the
network is dissociated. We refer to this model as the \emph{marginal beta model}; the
dissociatedness also follows immediately  from the following ``de Finetti type"
representation for the distribution of $X$:
\begin{equation}\label{eq:5}
    \mathbb{P}(X = x) =\int_{\mathbb{R}^n}\prod_{ij \in D(\nodeset)
    }\frac{e^{(\beta_i+\beta_j)x_{ij}}}{1+e^{(\beta_i+\beta_j)}}dF^{\nodeset}_B(\beta),
    \quad x \in \mathcal{G}_{\nodeset},
\end{equation}
where $F^{\nodeset}_B$ denotes the $|\nodeset|$-fold product distribution from
$F_B$.
The function $\phi$ in Proposition \ref{prop:1n} is in fact $p_{ij}$ as described in (\ref{eq:beta}). It clearly follows from  (\ref{eq:5}) that the marginal beta distribution is exchangeable;
in addition, (\ref{eq:betarep}) implies that the marginal beta model is summarized.

 \section{Consistent systems of network models}\label{sec:consistent}
In this section we shall consider systems of network models and their relation
to each other. We let $\nodeset$ be the basic set of network nodes and consider
models for finite subsets $\nodeset'\subseteq\nodeset$.  The
basic set $\nodeset$ may be finite or countably infinite.  If $P_{\nodeset'}$ is
a distribution of a random network $X_{\nodeset'}$ with nodeset $\nodeset'$ and
$\nodeset''\subseteq \nodeset'$ we let
$\Pi^{\nodeset''}_{\nodeset'}P_{\nodeset'}$ denote the distribution of the
induced subnetwork $X_{\nodeset''}$ under $P_{\nodeset'}$. Thus the
marginalization
operator
$\Pi^{\nodeset'}_{\nodeset''}$ takes as input a probability distribution on
$\mathcal{G}_{\nodeset'}$ and  returns the induced or marginal probability distribution over
$\mathcal{G}_{\nodeset''}$
corresponding to the dyads with endpoints in $\nodeset''$.
By constructions, marginalization operators satisfy the property
\[\Pi^{\nodeset'}_{\nodeset'''} =\Pi^{\nodeset''}_{\nodeset'''}\circ\Pi^{\nodeset'}_{\nodeset''}.\]
for any $\nodeset''' \subset \nodeset'' \subset \nodeset'$, so they form a
\emph{projective system}.
A collection $\{\mathcal{P}_{\nodeset'}, \nodeset'\subseteq \nodeset\}$ of random network models for all finite subsets $\nodeset'$ of $\nodeset$ is said to be (strongly) \emph{consistent} if for every $\nodeset''\subseteq \nodeset'$ in the collection it holds that
\[\mathcal{P}_{\nodeset''}=\Pi_{\nodeset''}^{\nodeset'}(\mathcal{P}_{\nodeset'}),\]
in other words it holds that the marginal to a subnetwork of any network model is identical to the corresponding network model for the subnetwork.
Thus a strongly consistent system forms a \emph{projective statistical field} in the sense of
\cite{lauritzen:88}. We say  the system is \emph{weakly consistent} if
\[\Pi_{\nodeset''}^{\nodeset'}(\mathcal{P}_{\nodeset'})\subseteq \mathcal{P}_{\nodeset''}\] i.e.\ marginalization of any network model for a larger nodeset is not in conflict with the network model for the subnetwork.

We first note that the bidirected Markov property itself is always \emph{marginalizable} in the  sense that if $X$ is bidirected Markov w.r.t.\ a graph $G$, then the marginal distribution of $X_A$ is automatically Markov w.r.t.\ the induced subgraph $G_A$, whereas the similar statement is generally not true for the undirected Markov property.

This holds in particular for the random network models when $G(\nodeset')=L_{\leftrightarrow}(\nodeset')$ is the bidirected incidence graph. In other words, \emph{the system of dissociated random network models is weakly consistent}. 
The same holds for the system of exchangeable random networks
$\{\mathcal{E}_{\nodeset'}, \nodeset'\subseteq \nodeset\}$, and when both
exchangeability and dissociatedness are combined to form the system of
exchangeable and dissociated random networks  $\{\mathcal{E}^{L_{\leftrightarrow}}_{\nodeset'},\nodeset'\subseteq \nodeset\}$, or when considering the system of models
$\{\mathcal{E}^{L^c_{\leftrightarrow}}_{\nodeset'},\nodeset'\subseteq \nodeset\}$, which are all Markov w.r.t.\ the bidirected complements of the line graphs.

Correspondingly, the M\"obius parametrization is  \emph{marginalizable} in the sense that if $X$ is a random network with nodeset $\nodeset$ and $\nodeset'\subseteq \nodeset$, the M\"obius parameters $z^{\nodeset'}_B$  for the induced subnetwork are   identical to the corresponding parameters $z^{\mathcal N}_B$ for $X$:
\begin{equation}\label{eq:marginalizable}z^{\nodeset'}_B=z^{\nodeset}_B=z_B,
    \mbox{ for all } B\in \mathcal{B}(\nodeset').\end{equation}

From any weakly consistent system $\{\mathcal{P}_{\nodeset'}, \nodeset'\subseteq \nodeset\}$  we can construct a strongly consistent system $\{\mathcal{P}^*_{\nodeset'}, \nodeset'\subseteq \nodeset\}$,  by letting
\[\mathcal{P}^*_{\nodeset'} = \bigcap_{\nodeset'':\nodeset''\supset\nodeset'}\Pi_{\nodeset'}^{\nodeset''}(\mathcal{P}_{\nodeset''})\]
provided that the intersection is non-empty. Thus each family in the system
consists exactly of those distributions that can be obtained as marginals of
distributions from models on arbitrarily large nodesets.  We shall thus say that the elements in the strongly consistent variant of a weakly consistent system are  \emph{completely extendable}.

We note that, for example, the system of exchangeable random networks is not strongly consistent. In Example~\ref{ex:exch}, the estimated distributions  are easily seen not to be extendable to a network with more than four nodes.
In general it is a  difficult issue to determine whether a given system of
M\"obius parameters $\{z^{\nodeset'}_B, B\in \mathcal {B}(\nodeset')\}$  is
extendable to a larger network with nodeset $\mathcal{N}$, i.e.\ if there
exists a system of M\"obius parameters $\{ z^{\nodeset}_B, B \in
\mathcal{B}(\nodeset)\}$, such that (\ref{eq:marginalizable}) holds.

Nevertheless, we can characterize the strongly
consistent sequences of exchangeable random networks as follows. Below, if $\nodeset$ is infinite, we let $\mathcal{B}(\nodeset)$ be the
set of all finite subsets of $D(\nodeset)$, which, as before, we represent as
finite edge-induced subgraphs of the complete graph on $\nodeset$.
\begin{prop}\label{prop:mobiusinfinite} All strongly consistent and infinitely
    extendable systems of exchangeable M\"obius parameters $( z_B, B \in
    \mathcal{B}(\nodeset) )$ are those which satisfy the conditions
 \begin{enumerate}[\rm (a)]
 \item $z_{\varnothing}=1$;
 \item $z_B = z_{B'} = z_U$ if $B$ and $B'$ are in $[U]$;
 \item   
    for all
    non-negative integers $n$ and all  $\nodeset' \subset \nodeset$ with $|
    \nodeset'|=n$
    \[
	\sum_{U\in \mathcal{U}_{\nodeset'}}(-1)^{|U|-|x|}r_U(x)z_U \geq 0, \quad \forall x
	\in \mathcal{G}_{\nodeset'},
    \]
\end{enumerate}
where we recall that $r_U(x)$ is the number of graphs in $[U]$ that contain $x$
as a subgraph.
In addition, the sequences of M\"{o}bius parameters corresponding to the
extremal distribution $\mathcal{E}_\infty$ 
are those satisfying the dissociated property that
if $B$ has connected components $C_1,\ldots,C_k$  then $z_B=z_{C_1}\cdots z_{C_k}$.
\end{prop}
\begin{proof}From (\ref{eq:marginalizable}) it follows that any system of valid
    M\"obius parameters is consistent and as previously argued, these are valid
    if and only if they satisfy $z_\varnothing=1$ and the non-negativity
    restriction in equation \eqref{eq:mobexch} above. We have also previously argued that the additional product condition characterizes dissociated networks which are the extreme points according to de Finetti's theorem (Proposition \ref{prop:1n}).\halm
\end{proof}
It is worth pointing out that, when $|\nodeset| = \infty$, if $X$ is a random
network on $\nodeset$ whose distribution  is extremal (and therefore
dissociated), then, for each non-empty
$U \in \mathcal{U}_{\nodeset'}$ with $|\nodeset'| < \infty$,
\begin{equation}\label{eq:zu.infty}
    z_U = \int_{[0,1]^{\nodeset'}} \prod_{ij \in E(U)} \phi(u_i,u_j) du,
\end{equation}
where $E(U)$ denotes the set of edges in $U$ and $\phi$ is some
measurable, symmetric function
from $[0,1]^2$ into $[0,1]$.

Using the arguments in Chapter 14 of  \cite{aldous:85} \citep[see also][for a
generalization]{kallenberg:05}, de Finetti's theorem (Proposition \ref{prop:1n})
and Proposition~\ref{prop:mobiusinfinite} yields the following
result.
\begin{coro}\label{cor:7}
The convex set of feasible M\"obius parameters are those which can be
represented as
\begin{equation}\label{eq:mixmobius}
z_B=\int_{[0,1]}\left\{\int_{[0,1]^n}\prod_{ij\in
B}\phi(u_i,u_j,\lambda)\,du\right\}\,d\lambda, \quad B \in
\mathcal{B}(\nodeset),
\end{equation}
where $n=|B|$ and 
$\phi$ is a (not necessarily unique) measurable function from $[0,1]^3$ to
$[0,1]$. The extreme points of this convex set correspond to the cases in
which the
function $\phi(u_i,u_j,\lambda)$ is a constant function of $\lambda$,  for almost all $\lambda\in [0,1]$.
\end{coro}
Still, these characterizations are quite implicit. It seems difficult to obtain more explicit characterizations. For example, Figure 16.1 on page 288 of \cite{lov12} displays the area of variation of
$(
z_{\n{\vcenter{\hbox{\includegraphics[scale=0.05]{figureex2-1.pdf}}}}},z_{\n{\vcenter{\hbox{\includegraphics[scale=0.05]{figureex2-4.pdf}}}}})$
as given in \eqref{eq:zu.infty} for all exchangeable and dissociated infnite
networks, and this is a very complicated region. The  set of infinitely extendable pairs is given by the convex hull of the region
\citep{bollobas:76}.

Using Corollary~\ref{cor:7} we can show that strongly consistent systems of
exchangeable models which are Markov w.r.t.\ bidirected complement of the
incidence graph must be Erd{\"o}s--Renyi models. First, we need a preliminary
lemma. The lemma and proof are analogues of Theorem 10 in  \cite{diaconis:freedman:81}.

\begin{lemma}\label{lem:erdosrenyi}Assume $X$ is an exchangeable random network
    with an infinitely extendable set of M\"obius parameters satisfying, for
    some $\eta \in (0,1)$,
\[
    z_{\n{\vcenter{\hbox{\includegraphics[scale=0.05]{figureex2-1.pdf}}}}}=\eta,\quad
    z_{\n{\vcenter{\hbox{\includegraphics[scale=0.05]{figureex2-2.pdf}}}}}=\eta^2,\quad
    z_{\n{\vcenter{\hbox{\includegraphics[scale=0.05]{figureex6-1.pdf}}}}}=\eta^4.\]
Then, the edges in $X$ are independent and identically distributed with $P(X_{ij}=1)=\eta$, i.e.\ $X$ is Erd\"os--Renyi.
\end{lemma}
\begin{proof}Since all elements are exchangeable, we have from (\ref{eq:mixmobius}) that
\[\eta=z_{\n{\vcenter{\hbox{\includegraphics[scale=0.05]{figureex2-1.pdf}}}}}=
\int_{[0,1]^3}\phi(u,v,\lambda)\,du\,dv\,d\lambda.\]
But also
\[z_{\n{\vcenter{\hbox{\includegraphics[scale=0.05]{figureex2-2.pdf}}}}}=\eta^2
=\int_{[0,1]^4}\phi(u,v,\lambda)\phi(u',v,\lambda)\,du\,du'\,dv\,d\lambda =
\int_{[0,1]^2}\left\{\int_{[0,1]}\phi(u,v,\lambda)\,du\right\}^2\,dv\,d\lambda.\]
Thus,
\[\int_{[0,1]^2}\left\{\int_{[0,1]}\phi(u,v,\lambda)\,du-\eta\right\}^2\,dv\,d\lambda=0,\]
and hence
\[\int \phi(u,v,\lambda)\,du =\eta\quad \mbox{for almost all $(v,\lambda)$},\]
which in turn implies that, if we assume $\nodeset=\{1,2,3,\ldots\}$,
\[z_{\n{\vcenter{\hbox{\includegraphics[scale=0.05]{figureex2-3.pdf}}}}}=P(X_{12}=1,X_{34}=1)=\int_{[0,1]^5}\phi(u,v,\lambda)\phi(u',v',\lambda)\,du\,du'\,dv\,\,dv'\,d\lambda=\eta^2.\]
Next, we have that
\begin{eqnarray*}\eta^4=z_{\n{\vcenter{\hbox{\includegraphics[scale=0.05]{figureex6-1.pdf}}}}}&=&
\int_{[0,1]^5}\phi(u,v,\lambda)\phi(u',v,\lambda)\phi(u,v',\lambda)\phi(u',v',\lambda)\,du\,du'\,dv\,dv'\,d\lambda\\&=&
\int_{[0,1]^3}\left\{\int_{[0,1]}\phi(u,v,\lambda)\phi(u',v,\lambda)\,dv\right\}^2\,du\,du'\,d\lambda,\end{eqnarray*}
and hence we conclude as above that, for almost all $(u,u',\lambda)$,
\begin{equation}
    \label{eq:dv}
    \int_{[0,1]}\phi(u,v,\lambda)\phi(u',v,\lambda)\,dv=\eta^2.
\end{equation}
We would now wish to let $u=u'$ in the above equation but this might just be the exceptional set. However, we may
 extend $\phi$ periodically to be defined on $(-\infty,\infty)\times [0,1]^2$  by letting
 \[\phi(u,v,\lambda)=\phi(u\,\text{mod}\, 1,v,\lambda)\]so that \eqref{eq:dv}
 still holds almost surely. Then, using Lebesgue's theorem, we can write
\begin{eqnarray*}\int_{[0,1]}\phi(u,v,\lambda)^2\,dv&=&
\int_{[0,1]}\lim_{\epsilon\to 0} \frac{1}{4\epsilon^2}\int_{-\epsilon}^\epsilon
\int_{-\epsilon}^\epsilon \phi(u+s,v,\lambda)\phi(u+t,v,\lambda)\,ds\,dt\,dv\\&=&\lim_{\epsilon\to 0} \frac{1}{4\epsilon^2}
\int_{-\epsilon}^\epsilon \int_{-\epsilon}^\epsilon
\int_{[0,1]}\phi(u+s,v,\lambda)\phi(u+t,v,\lambda)\,dv\,ds\,dt=\eta^2. \end{eqnarray*}
As above, we conclude, for almost all $\lambda, u,v$, that
$\phi(u,v,\lambda)=\eta$ and hence we have $z_D=\eta^{|D|}$, implying that all
edges are independent. Hence, the model is Erd\"os--Renyi.
\halm\end{proof}

\begin{prop}\label{prop:bi_comp_kneser}
Consider a countable nodeset $\nodeset$
    and a strongly consistent family $\{ \mathcal{P}_{\nodeset'},
\nodeset'\subseteq \nodeset\}$ of random network models; if all elements of
$\mathcal{P}_{\nodeset'}$ are exchangeable and Markov w.r.t.\ the complement of
the bidirected line graph $L^c_{\leftrightarrow}(\nodeset')$, then  they are all Erd{\"o}s--Renyi models.
\end{prop}
\begin{proof}
		We assume that $\nodeset=\{1,2,3,\ldots\}$.
    By exchangeability, we have that, for some $\eta \in [0,1]$,
$z_{\n{\vcenter{\hbox{\includegraphics[scale=0.05]{figureex2-1.pdf}}}}} =
\mathbb{P}(X_{12}=1) = \eta$.
Since $X_{12}\cip X_{23}$, we also  have that
\[z_{\n{\vcenter{\hbox{\includegraphics[scale=0.05]{figureex2-2.pdf}}}}}=\mathbb{P}(X_{12}=X_{23}=1)=\mathbb{P}(X_{12}=1)\mathbb{P}(X_{23}=1)=\eta^2=z_{\n{\vcenter{\hbox{\includegraphics[scale=0.05]{figureex2-3.pdf}}}}},\]
where the last equality is obtained as in the proof of Lemma~\ref{lem:erdosrenyi}.
The bidirected complement of the
incidence graph for $\nodeset'=\{1,2,3,4\}$ contains the edges \[ \{12\leftrightarrow34, 13\leftrightarrow 24, 14\leftrightarrow 23\}.\]
Since all the distributions are Markov w.r.t.\ this graph we  have
\[z_{\n{\vcenter{\hbox{\includegraphics[scale=0.05]{figureex6-1.pdf}}}}}=P(X_{12}=1,X_{34}=1)P(X_{14}=1,X_{23}=1)=(z_{\n{\vcenter{\hbox{\includegraphics[scale=0.05]{figureex2-3.pdf}}}}})^2=\eta^4.\]
and the conclusion now follows from Lemma~\ref{lem:erdosrenyi}.
\halm\end{proof}

Analogous results hold true if the Markov assumption in
Proposition~\ref{prop:bi_comp_kneser} is replaced with Markov properties with
respect to
the undirected incidence graph as in the Frank--Strauss models or the undirected
complement of the line graph, as in the Kneser models. The proof of these facts
are quite elementary and we abstain from giving the details. In fact these
sequences are not even weakly consistent, as marginals of undirected Markov
distributions are not Markov with respect to the corresponding induced subgraph.

Finally we hasten to point out that indeed the \emph{system of marginal beta models is strongly consistent} by construction and that this clearly remains true if  the mixing distribution $F$ is assumed to belong to any specific subset of distributions $\mathcal{F}$.

\section{Discussion}
We have derived a complete characterization of possible Markov
properties of exchangeable network models, and studied some of the implied models.
We also consider summarized exchangeable network models in which the probability
of a network is only a function of its degree distributions, and relate them to
the other models we discuss in the paper.
Overall, our findings have unveiled several interesting properties of
exchangeable network
models and established connections among seemingly disparate concepts used in 
network analysis, the probability literature on exchangeable arrays and the discrete
mathematics literature on graph limits.

While we have focused on exchangeable
and extendable network models, we point out that these are by no means the only interesting network models. In fact, many other
parametric
models, not fulfilling those
properties,  can also be successfully deployed to represent and study networks. Examples include
the beta model \citep{chatterjee:diaconis:sly:11,Yan,rin13} and the sparse
Bernoulli model considered in \cite{KrKo15q}; see also the remarks in the introduction on partial conditional independence and  on  weakened forms of exchangeability. 

We conclude with a final remark on the computational difficulties associated
with
fitting the various models discussed in the paper. As it is very often the case with ERGMs,
model fitting can be rather challenging, and our models are no exception.
For the DE and DSE models, one
could in principle rely on the algorithms proposed in \cite{drt08},
appropriately modified to allow for zero counts and for the fact that the MLE of
the model parameters is nearly all cases on the boundary of the parameter space.
However such algorithms may not
scale well with the network size.
As for the model given in  \eqref{eq:expexch}, and its submodels, the SE and
SE* and the EM models, parameter estimation can be carried out using
 pseudo-maximum likelihood and MCMC-based methods,
 \citep[see, e.g.][]{CIS-93373,geyer.thompson:1992}, implemented, for instance, in the {\tt
 R} package {\tt ergm} \citep{ergm.package}. In all these cases, however, we
 expect that fitting could be problematic,  as these
procedures do not scale well with the network size and MCMC convergence can be
very slow or hard to assess. In particular, we expect that the  issues of
degeneracy \citep[see][]{Schweinberger,rinaldo2009} plaguing many ERGMs will
also impact estimation of these models.

\section*{Acknowledgements}
The authors are grateful to \'{E}va Czabarka for proving an earlier version of
Lemma~\ref{lem:onlyif}.
Alessandro Rinaldo and Kayvan Sadeghi were partially supported by AFOSR
grant FA9550-14-1-0141.

\end{document}